\documentclass[10pt]{amsart}
\usepackage[latin9]{inputenc}
\usepackage{verbatim}
\usepackage{amsthm}
\usepackage{amstext}
\usepackage{amssymb}
\usepackage{esint}
\usepackage[unicode=true,pdfusetitle,bookmarks=true,bookmarksnumbered=false,bookmarksopen=false,breaklinks=false,pdfborder={0 0 1},backref=false,colorlinks=false]{hyperref}

\makeatletter
\numberwithin{equation}{section}
\numberwithin{figure}{section}
\theoremstyle{plain}
\newtheorem{thm}{\protect\theoremname}
  \theoremstyle{definition}
  \newtheorem{defn}[thm]{\protect\definitionname}
  \theoremstyle{remark}
  
  \theoremstyle{plain}
  \newtheorem{lem}[thm]{\protect\lemmaname}

\usepackage{amsthm}\usepackage{amsfonts}\usepackage{amscd}\usepackage{url}\usepackage{color}\usepackage[matrix,arrow]{xy}\usepackage{mathrsfs}\usepackage{enumerate}\usepackage{amsopn}
\usepackage{bbm}
\usepackage{cite}\allowdisplaybreaks[1]
\usepackage{bbm}\usepackage{stmaryrd}\usepackage{mdef}
\date{\today}

\newcommand{\no}{\noindent}
  \providecommand{\definitionname}{Definition}
  \providecommand{\lemmaname}{Lemma}
  
  \providecommand{\remarkname}{Remark}
\providecommand{\theoremname}{Theorem}

\makeatother

  \providecommand{\definitionname}{Definition}
  \providecommand{\lemmaname}{Lemma}
  \providecommand{\remarkname}{Remark}
\providecommand{\theoremname}{Theorem}

\definecolor{gray}{rgb}{0.75, 0.75, 0.75}

\begin{document}

\title[Stochastic parabolic-hyperbolic equations]{Stochastic non-isotropic degenerate parabolic-hyperbolic equations}
\begin{abstract} We introduce the notion of pathwise entropy solutions for a class of degenerate parabolic-hyperbolic equations with non-isotropic nonlinearity and fluxes with rough time dependence and prove their well-posedness. In the case of Brownian noise and periodic boundary conditions, we prove that the pathwise entropy solutions converge to their spatial average and provide an estimate on the rate of convergence. The third main result of the paper is a new regularization result in the spirit of averaging lemmata. This work extends both the framework of pathwise entropy solutions for stochastic scalar conservation laws introduced by Lions, Perthame and Souganidis and the analysis of the long time behavior of stochastic scalar conservation laws by the authors to a new class of equations.
\end{abstract}

\author[B. Gess]{Benjamin Gess}
\address{Max-Planck Institute for Mathematics in the Sciences \\
04103 Leipzig\\
Germany }
\email{bgess@mis.mpg.de}

\author{Panagiotis E. Souganidis }

\address{Department of Mathematics \\
University of Chicago \\
Chicago, IL 60637 \\
USA }

\email{souganidis@math.uchicago.edu }



\maketitle

\section{Introduction}

\noindent We are interested in ``stochastic'', scalar, degenerate parabolic-hyperbolic equations of the form 
\begin{equation}
\begin{cases}
du+{\displaystyle \sum_{i=1}^{N}\partial_{x_{i}}F^{i}(u)\circ dz^{i}=\div(A(u)Du)dt\quad\text{ in }\quad\R^{N}\times(0,T),}\\[1.5mm]
u=u^{0}\quad\text{ on }\quad\R^{N}\times\{0\},
\end{cases}\label{eq:scl}
\end{equation}
where 
\begin{equation}\label{takis data0}
u_{0}\in(BV\cap L^{\infty})(\R^{N}),  \ z\in C_{0}([0,T];\R^{N}) \  \text{and} \  A\in L_{loc}^{\infty}(\R;\R^{N\times N}).
\end{equation}
\noindent A particular example of the signal $z$ is the $N$-dimensional (fractional) Brownian motion. The Stratonovich notation $\circ$ in \eqref{eq:scl} is justified by showing that the pathwise entropy solution constructed in this paper is the limit of solutions $u^{(l)}$ to \eqref{eq:scl} with $z$ replaced by smooth approximations $z^{(l)}$ converging to $z$.

\noindent We assume that 
\begin{equation}\label{takis data1}
A\in  L_{loc}^{\infty}(\R; \mathcal S^N) \ \text{ is nonnegative},
\end{equation}
where $\mathcal S^N$ is the space of symmetric $N\times N$ matrices, and 
\begin{equation}\label{takis data2}
F\in C^{2}_{loc}(\R;\R^{N}) \ \text{with} \  F'(0)=0.
\end{equation}

\noindent  In this paper we introduce the notion of pathwise entropy solutions for \eqref{eq:scl} and prove that it is well-posed. In the setting of stochastic scalar conservation laws, that is when $A\equiv0$ in \eqref{eq:scl}, the notion of pathwise entropy solutions was introduced by  Lions, Perthame and Souganidis \cite{LPS13}, while Chen and Perthame \cite{CP03} studied the entropy solutions of the ``deterministic'' version of \eqref{eq:scl}, that is for $z(t)=(t,\dots,t)$. We also refer to Perthame and Souganidis \cite{PS03} which introduced an equivalent solution in the deterministic case called dissipative solution and to Lions, Perthame and Souganidis \cite{LPS14}, Gess and Souganidis \cite{GS14}, Friz and Gess \cite{FG14}, Hofmanova \cite{H16} for stochastic scalar conservation laws with $A\equiv0$ allowing for spatially inhomogeneous flux and semilinear noise. We note that, as compared to \cite{LPS13,LPS14,GS14,H16}, the inclusion of the non-isotropic parabolic part $A$ in \eqref{eq:scl} leads to additional complications and additional approximation and localization steps in the proof of well-posedness are needed. 

\noindent As particular examples, \eqref{eq:scl} includes stochastic porous media equations, that is 
\begin{equation}
du+{\displaystyle \sum_{i=1}^{N}\partial_{x_{i}}F_{i}(u)\circ dz^{i}=\D u^{[m]}dt\quad\text{ in }\quad\R^{N}\times(0,T)}  \ \text{ where $m>1$ and $u^{[m]}:=|u|^{m-1}u$.}\label{eq:intro_SPME}
\end{equation}
In this case our results extend previous work on stochastic porous media equations, based on the variational approach to SPDE, which did not allow treatment of nonlinear transport noise as in \eqref{eq:intro_SPME}. See, for example, Pr\'evot and Röckner~\cite{PR07}, Krylov and Rozovski{\u\i} \cite{KR79}, Pardoux \cite{P75}, Barbu, Da Prato and Röckner~\cite{BDPR09-2,Barbu2009a,BDPR08} and the references therein. In addition, in contrast to previous work, the pathwise approach to \eqref{eq:scl} developed in this paper immediately implies the existence of a corresponding random dynamical system. This is of particular interest, since \eqref{eq:scl} contains nonlinear multiplicative noise for which the existence of a random dynamical system is known to be a difficult problem (see, for example, Mohammed, Zhang and Zhao~\cite{MZZ08}, Flandoli~\cite{F95} and Gess~\cite{G14}).

In the second part of the paper we study  the long-time behavior and regularity properties of \eqref{eq:scl} set  on the torus $\TT^{N}$ and driven by Brownian motion, that is the initial value problem 
\begin{equation}
\begin{cases}
du+{\displaystyle \sum_{i=1}^{N}\partial_{x_{i}}F^{i}(u)\circ d\b_{t}^{i}=\div(A(u)Du)dt\quad\text{ in }\quad\TT^{N}\times(0,\infty),}\\[1.5mm]
u=u_{0}\quad\text{ on }\quad\TT^{N}\times\{0\},
\end{cases}\label{eq:scl-torus}
\end{equation}
with 
\begin{equation}\begin{split}\label{takis data3}
A\in C(\R;\mathcal S^N)\cap C^{1}(\R\setminus\{0\};\mathcal S^N),\\
|F''(\xi)|+|A'(\xi)|\le C(1+|\xi|^{p_{1}}+|\xi|^{p_{2}})
\end{split}
\end{equation}
for some $C>0,p_{1},p_{2}\in(-1,\infty)$ and all $\xi\in\R\setminus\{0\}$,
\begin{equation}
{\b}=(\b^{1},\ldots,\b^{N})\ \text{ is a standard two-sided Brownian motion},\label{b}
\end{equation}
and 
\begin{equation}
u_{0}\in L^{1}(\TT^{N}).\label{id}
\end{equation}
We prove that, under a genuine nonlinearity condition for $F,A$, the solutions to \eqref{eq:scl-torus} converge to their spatial average and we provide a rate of convergence. In the deterministic setting, that is when $z(t)=(t,\dots,t)$, Chen and Perthame~\cite{CP09} proved the convergence to the spatial average by different methods and without an estimate on the rate of convergence. In the hyperbolic case, that is when  ~$A\equiv0$ and restricted to one and two space dimensions respectively, Lax~\cite{L57}, E and Engquist~\cite{EE93} provided estimates on the rate of convergence. Convergence to the spatial average under weak conditions on the flux $F$ has been shown by Dafermos~\cite{D13}. A rate of convergence in any dimension for both deterministic and stochastic scalar conservation laws ($A\equiv0$) was established by the authors in  \cite{GS14-2}.

The third  result presented here is a regularity estimate based on averaging techniques. This extends the regularity results developed in \cite{GS14-2} for the hyperbolic case and provides new regularity estimates for stochastic porous media equations. Averaging lemmata for deterministic parabolic-hyperbolic equations have been established by Tadmor and Tao~\cite{TT07}.  

Stochastic scalar conservation laws driven by multiplicative semilinear noise, that is, 
\begin{equation}\label{eq:semilinear_SSCL}
   du + \div F(u) dt = g(u)dW_t
\end{equation}
have been intensively studied in recent years; we refer to Holden and Risebro \cite{HR97}, Feng and Nualart \cite{FN08}, Debussche and Vovelle \cite{DV10}, Chen, Ding and Karlsen \cite{CDK12}, Bauzet, Vallet and Wittbold \cite{BVW13}, Hofmanova \cite{H13} and the references therein. In particular, Debussche, Hofmanova and Vovelle \cite{DHV13} and Bauzet, Vallet and Wittbold \cite{BVW15} studied SPDE of the type \eqref{eq:semilinear_SSCL} additionally including a second order quasilinear term. The long-time behavior of solutions to \eqref{eq:semilinear_SSCL} with additive noise has been analyzed by E, Khanin, Mazel and Sinai \cite{EKMS00} and Debussche and Vovelle \cite{DV13}.

\subsection*{Organization of the paper}
In Section \ref{sec:Main-results}, we introduce the notion of pathwise entropy solutions and state the results about its well-posedness for general continuous paths and the long time behavior with rates as well as regularity in the stochastic case. The proofs are given in Sections \ref{sec:unique}, \ref{sec:existence}, \ref{sec:ltb}, \ref{sec:reg} respectively. Since some of the results are technical, we have chosen to present in the first two appendices the proofs of the most technical estimates. Finally, for the convenience of the reader we recall in Appendix~\ref{takis kinetic100}  the definition \cite[Definition 2.2]{CP03} of the kinetic solution to \eqref{eq:scl} for smooth driving signals. 

\subsection*{Notation}
We work in $\R^N$ or on the torus $\TT^N$, $S^{N-1}$ is  the unit sphere in $\R^N$, $\R_+:=(0,\infty)$ and $\delta$ is the ``Dirac'' mass at the origin in $\R$.   For notational convenience, we write  $A\lesssim B$ if $A\leq CB$ for some $C>0$, and $A\sim B$,  if $A\lesssim B$ and $B\lesssim A$. For a matrix $A\in\R^{N\times N}$ we write $A=(a_{ij})_{i,j=1}^{N}$ and, given  two matrices $A,B$, we set $A:B:=\sum_{i,j=1}^{N}a_{ij}b_{ij}.$ The subspace of $L^{1}$-functions with  bounded total variation is $BV$ and  its norm is $\|\cdot \|_{BV}$. If $f \in BV$, $BV(f)$ is its total variation. We  write $\delta_{c}$ for  the ``Dirac mass'' measure in $L^{1}(\TT^{N})$ charging the constant function $c\in\R$ and we set $L_{c}^{1}(\TT^{N}):=\{u\in L^{1}(\TT^{N}):\, u\text{ has average }c\}$.
We set $C_{0}([0,T];\R^{N}):=\{z\in C([0,T];\R^{N}):\ z(0)=0\}$. For an open set $\mcO\subseteq\R^{N}$, $C_{c}^{1}(\mcO)$ is the space of all continuously differentiable functions with compact support in $\mcO$. If $F:[0,T]\to\R$,  $F|_{s}^{t}:=F(t)-F(s)$ for all $s,t\in[0,T]$.  In the second part of the paper, we will omit, when it does not cause any confusion, the dependence in $\omega$ and we occasionally write $m(x,\xi,s)dxd\xi ds$ instead of $dm(x,\xi,s).$ The space of homogeneous Bessel potentials,  $W^{\l,p}$ for $\l>0$, $p\in[1,\infty)$,  is 
\[
W^{\l,p}:=\{f\in L^{p}(\TT^{N}):\ (|n|^{\l}\hat{f}(n))^{\vee}\in L^{p}(\TT^{N})\},
\]
where $\hat{f}$ is the discrete Fourier transform of $f$ on $\TT^{N}$ and $f^{\vee}$ is its inverse. The homogeneous Bessel potential spaces coincide with the domains of the fractional Laplace operators $(-\D)^{\frac{\l}{2}}$ on $L^{p}(\TT^{N})$. For notational simplicity we set $H^{\l}:=W^{\l,2}$.

\section{The main results\label{sec:Main-results}}
\subsection*{Pathwise entropy solutions}\label{sub:Pathwise entropy solutions}
We  introduce here the notion of pathwise entropy solutions to \eqref{eq:scl}, which is based on the kinetic formulation of \eqref{eq:scl} and on choosing test-functions transported along the characteristics. For (hyperbolic) scalar conservation laws this  was introduced in \cite{LPS13} and it was motivated from the theory of stochastic viscosity solutions for fully nonlinear first- and second-order PDE including stochastic Hamilton-Jacobi equations developed by Lions and  Souganidis (see \cite{LS98,LS98-2,LS02,LS00-2,LS00}). The theory was extended to the spatially inhomogeneous case by Lions, Perthame and Souganidis \cite{LPS14}, Gess and Souganidis \cite{GS14-2} and Hofmanova \cite{H16}. In view of the parabolic part in \eqref{eq:scl}, the kinetic formulation has to be appropriately adapted;  in the deterministic case this was done in \cite{CP03}.

For smooth driving paths $z\in C^{1}([0,T];\R^{N})$ the well-established theories of entropy solutions and kinetic solutions apply to \eqref{eq:scl}.  In order to fix  the ideas and the notation,  it is necessary to introduce the functions
\begin{equation}
\chi(u,\xi):=\begin{cases}
+1 & \text{if }0\leq\xi\leq u,\\
-1 & \text{if }u\leq\xi\leq0,\\
0 & \text{otherwise},
\end{cases}\label{eq:char_fctn-1}
\end{equation}
and, given $u:\R^N \times [0,\infty) \to \R$,  
\begin{equation}
\chi(x,\xi,t):=\chi(u(x,t),\xi).\label{eq:char_fctn}
\end{equation}
and to  observe that 
\[
\partial_{\xi}\chi(x,\xi)=\d(\xi)-\d(\xi-u). 
\]
Finally, to simplify the notation, we use the functions 
\begin{align*}
f_{i}(\xi,t) & :=(F^{'})^{i}(\xi)\dot{z}^{i}(t).
\end{align*}
The kinetic form of \eqref{eq:scl} then is 
\begin{align}
\begin{cases}
\partial_t\chi+{\displaystyle f(\xi,t)\cdot D_{x}\chi}-\sum_{i,j=1}^{d}a_{ij}(\xi)\partial_{x_{i}x_{j}}^{2}\chi=\partial_{\xi}(m+n) \ \text{ in } \ \R^N\times\R\times [0,\infty),\\[1mm]
\chi(x,\xi,0)=\chi(u_{0}(x),\xi),
\end{cases}\label{eq:kinetic_form}
\end{align}
where the entropy defect measure $m$ and the parabolic dissipation measure $n$ are nonnegative, bounded measures in $\R^{N}\times\R\times[0,T]$  for each $T>0$. In the sequel, when we do not need to differentiate between $m$ and $n$, we set $q:=m+n$.

The notion of kinetic solutions is not well defined for rough driving signals $z$ that are merely continuous, since the coefficients $f(\xi,t)$ blow up with $\dot{z}$. On the other hand, following \cite{LPS13,LPS14,GS14-2} we observe that the linearity of \eqref{eq:kinetic_form} in $\chi$ suggests that we may use the characteristics of \eqref{eq:kinetic_form} to derive a stable notion of solutions, which will be the pathwise entropy solutions. 

We continue assuming $z\in C^{1}([0,T];\R^{N})$ and  consider the transport equation 
\begin{equation}
\begin{cases}
\partial_{t}\vr+{\displaystyle f(\xi,t)\cdot D_{x}\vr}=0\ \text{ in }\ \R^{N}\times\R\times[0,T],\\[1mm]
\vr_{0}=\vr^{0}\ \text{ on }\ \R^{N}\times\R\times\{0\}.
\end{cases}\label{eq:transport}
\end{equation}
 For each $(y,\eta)\in\R^{N+1}$ and $\vr^{0}\in C_{c}^{\infty}(\R^{N}\times\R)$,  $\vr=\vr(x,y,\xi,\eta,t)$ is the solution to \eqref{eq:transport} with 
\begin{equation}
\vr(\cdot,y,\cdot,\eta,0)=\vr^{0}(\cdot-y,\cdot-\eta)=\vr^{s,0}(\cdot-y)\vr^{v,0}(\cdot-\eta)\label{eq:vr_product_form}
\end{equation}
 and $\vr^{s,0}\in C_{c}^{\infty}(\R^{N})$, $\vr^{v,0}\in C_{c}^{\infty}(\R)$. Here, we use the superscripts $s$ and $v$ to emphasize that the functions act on the space and velocity variables respectively.
 
The convolution along characteristics  is then given by 
\[
\vr\ast\chi(y,\eta,t):=\int\vr(x,y,\xi,\eta,t)\chi(x,\xi,t)dxd\xi,
\]
and it follows from Lemma~\ref{lem:kinetic-pathwise} below that, for all $ $$(y,\eta)\in\R^{N+1}$ and in the sense of distributions in $t\in [0,T]$,
\begin{align*}
	\partial_{t}\vr\ast\chi(y,\eta,t)=
	&\sum_{i,j=1}^{N}\int a_{ij}(\xi)\chi(x,\xi,t)\partial_{x_{i}x_{j}}^{2}\vr(x,y,\xi,\eta,t)dxd\xi\\
	&-\int\partial_{\xi}\vr(x,y,\xi,\eta,t)q(x,\xi,t)dxd\xi\ \ \text{in }\R^{N}\times\R\times(0,T).
\end{align*}

The solution of \eqref{eq:transport} can be expressed in terms of the associated (backward) characteristics. Due to the spatial homogeneity of the flux $f$, the characteristics starting at $(y,\xi)$ and the solution $\vr=\vr(x,y,\xi,\eta,t)$ to \eqref{eq:transport} are given respectively by the explicit form
\[
{Y}_{(y,\eta)}^{i}(t)=y+f^{i}(\xi){z}^{i}(t),
\]
and
\begin{align}
\vr(x,y,\xi,\eta,t) 
& = \vr^{s}(x,y,\xi,t)\vr^{v,0}(\xi-\eta)\label{eq:test-functions}
 :=\vr^{s,0}(x-f(\xi){z}(t)-y)\vr^{v,0}(\xi-\eta). 
\end{align}
Note that, in contrast to \eqref{eq:transport}, the characteristic equations are trivially well defined also for rough driving signals $z\in C_{0}([0,T];\R^{N})$. 

Since $A(u)$ is a symmetric, nonnegative matrix, it has a square root, that is there exists $\s(u)=(\s_{ij}(u))_{i,j=1}^{N}:=A(u)^{\frac{1}{2}}\in L_{loc}^{\infty}(\R;\R^{N\times N})$, such that
\[
a_{ij}(u)=\sum_{k=1}^{N}\s_{ik}(u)\s_{jk}(u).
\]
For notational convenience we write
\[
\b_{ik}'(u)=\s_{ik}(u).
\]

Next we give the definition of pathwise entropy solutions.
\begin{defn}
\label{def:path_e-soln-2}Let $u_{0}\in(L^{1}\cap L^{\infty})(\R^{N})$ and $T>0$. A function $u\in C([0,T];L^{1}(\R^{N}))\cap L^{\infty}(\R^{N}\times[0,T])$ is a pathwise entropy solution to \eqref{eq:scl}, if 
\begin{enumerate}
\item[(i)]
for all $k\in\{1,\dots,N\}$
\begin{equation}
\sum_{i=1}^{N}\partial_{x_{i}}\b_{ik}(u)\in L^{2}(\R^{N} \times [0,T]),\label{eq:b-regularity}
\end{equation}
\item[(ii)] for all $(y,\eta)\in\R^{N+1}$ and all  $k\in\{1,\dots,N\}$
\begin{equation}
\sum_{i=1}^{N}\int_{\R^{N+1}}\chi(x,\xi,t)\sigma_{ik}(\xi)\partial_{x_{i}}\vr(x,y,\xi,\eta,t)dxd\xi=-\sum_{i=1}^{N}\int_{\R^{N}}\partial_{x_{i}}\b_{ik}(u(x,t))\vr(x,y,u(x,t),\eta,t)dx,\label{eq:chain-rule}
\end{equation}
\item[(iii)] there exists a non-negative bounded measure $m$ on $\R^{N}\times\R\times[0,T]$ such that, for all test functions $\vr$ given by \eqref{eq:transport} with $\vr^{0}\in C_{c}^{\infty}(\R^{N+1})$ as in \eqref{eq:vr_product_form}, for all $(y,\eta)\in\R^{N+1}$ and in the sense of distributions in $t\in[0,T]$,
\begin{equation}\label{eq:kinetic}
\frac{d}{dt}(\vr\ast\chi)(y,\eta,t)=  \sum_{i,j=1}^{N}\int\chi(x,\xi,t)a_{ij}(\xi)\partial_{x_{i}x_{j}}^{2}\vr(x,y,\xi,\eta,t)dxd\xi  -\int\partial_{\xi}\vr(x,y,\xi,\eta,t)(m+n)(x,\xi,t)dxd\xi,
\end{equation}
where $n$ is the non-negative measure on $\R^{N}\times\R\times[0,T]$ given by
\begin{equation}\label{eq:diss_meas}
n(x,\xi,t):=\d(\xi-u(t,x))\sum_{k=1}^{N}(\sum_{i=1}^{N}\partial_{x_{i}}\b_{ik}(u(t,x)))^{2}. 
\end{equation}
\end{enumerate}
\end{defn}
In the sequel, we refer to \eqref{eq:b-regularity} and \eqref{eq:chain-rule} as the regularity and chain rule properties respectively.

When $A\equiv0$, that is for the hyperbolic problem,  \eqref{eq:kinetic} alone was used as the definition of a pathwise entropy solution in \cite{LPS13}. In the presence of $A$ it is, however, necessary to also include the chain-rule \eqref{eq:chain-rule} in the definition, a fact already observed in the deterministic case in \cite{CP03}. 

The definition can be reformulated as follows: If $u_{0},m$ and $n$ are  as in Definition \ref{def:path_e-soln-2},  then $u\in C([0,T];L^{1}(\R^{N}))\cap L^{\infty}([0,T]\times\R^{N})$  is a pathwise entropy solution to \eqref{eq:scl} if and only if, instead of \eqref{eq:kinetic}, for all $(y,\eta)\in\R^{N+1},s\le t$ and all test functions $\vr$ given by \eqref{eq:transport} with $\vr^{0}\in C_{c}^{\infty}(\R^{N+1})$, 
\begin{equation}\label{eq:gen_kinetic_integrated-1-2}
\begin{cases}
\vr\ast\chi(y,\eta,t)-\vr\ast\chi(y,\eta,s)=& \sum_{i,j=1}^{N}\int_{s}^{t}\int\chi(x,\xi,r)a_{ij}(\xi)\partial_{x_{i}x_{j}}^{2}\vr(x,y,\xi,\eta,s)d\xi dxdr\\
&  -\int_{s}^{t}\int\partial_{\xi}\vr(x,y,\xi,\eta,r)(m+n)(x,\xi,r)dxd\xi dr.
\end{cases}
\end{equation}
\no Next we show that, for smooth paths, the notions of kinetic and pathwise entropy solutions are equivalent. For the convenience of the reader we recall from \cite{CP03} the definition of the kinetic solution in Appendix~\ref{takis kinetic100}. 
\begin{lem}
\label{lem:kinetic-pathwise}Assume that $z\in C^{1}_0([0,T];\R^{N})$ and $u\in C([0,T];L^{1}(\R^{N}))\cap L^{\infty}([0,T]\times\R^{N})$. Then $u$ is a pathwise entropy solution to \eqref{eq:scl} if and only if it is a kinetic solution.
\end{lem}
\begin{proof}
Let $u$ be a kinetic solution to \eqref{eq:scl}. Since $u\in L^{\infty}([0,T]\times \R^{N})$, we may take $\psi\equiv1$ in \eqref{takis cp1} and $\psi_{2}\equiv1$ in \eqref{takis cp2}. 

Then, for all $\vp\in C_{c}^{1}(\R^{N}\times\R\times(0,T)),$ 
\[
\int_{0}^{T}\int\chi\big(\partial_{t}\vp+f(\xi,r)\cdot D_{x}\vp+\sum_{i,j=1}^{N}a_{ij}(\xi)\partial_{x_{i}x_{j}}^{2}\vp\big)d\xi dxdr=\int_{0}^{T}\int q\partial_{\xi}\vp d\xi dxdr.
\]
Since $u\in C([0,T];L^{1}(\R^{N}))$ it follows that, for all $s\le t$ and all $\vp\in C_{c}^{1}(\R^{N}\times\R\times[0,T]),$ 
\begin{align*}
 & \int_{s}^{t}\int\chi(x,\xi,r)\big(\partial_{t}\vp(x,\xi,r)+f(\xi,r)\cdot D_{x}\vp(x,\xi,r)+\sum_{i,j=1}^{N}a_{ij}(\xi)\partial_{x_{i}x_{j}}^{2}\vp(x,\xi,r)\big)d\xi dxdr\\
 & -\int\chi(x,\xi,t)\vp(x,\xi,t)d\xi dx+\int\chi(x,\xi,s)\vp(x,\xi,s)d\xi dx=\int_{s}^{t}\int q\partial_{\xi}\vp d\xi dxdr.
\end{align*}
Choosing $\vp=\vr$ yields 
\begin{align*}
 & -\int\chi(x,\xi,\cdot)\vr(x,y,\xi,\eta,\cdot)d\xi dx\Big|_{s}^{t}+\sum_{i,j=1}^{N}\int_{s}^{t}\int\chi(x,\xi,r)a_{ij}(\xi)\partial_{x_{i}x_{j}}^{2}\vr(x,y,\xi,\eta,s)d\xi dxdr=\int_{s}^{t}\int q\partial_{\xi}\vr d\xi dxdr,
\end{align*}
which in view of \eqref{eq:gen_kinetic_integrated-1-2} implies that $\chi$ is pathwise entropy solution to \eqref{eq:scl}. 

Next we derive the chain-rule. It follows from \eqref{takis cp2} that, for all $\psi\in C_{c}^{\infty}$, 
\begin{equation}
\sqrt{\psi(u(x,t))}\sum_{i=1}^{N}\partial_{x_{i}}\b_{ik}(u(x,t))=\sum_{i=1}^{N}\partial_{x_{i}}\b_{ik}^{\psi}(u(x,t))\quad\text{a.e. in }(t,x).\label{eq:det_chain}
\end{equation}
where 
\begin{equation}\label{eq:b_psi}
  (\b_{ik}^{\psi})'(u)=\s_{ik}(u)\sqrt{\psi(u)},
\end{equation}
and thus
\[
\b_{ik}^{\psi}(u(x,t))=\int\chi(x,\xi,t)\s_{ik}(\xi)\sqrt{\psi(\xi)}d\xi.
\]
Choose $\sqrt{\psi(\cdot)}:=\vr^{v,0}(\cdot-\eta)\vr_{\d}(\cdot-\xi)$ in \eqref{eq:det_chain}, where $\vr_{\d}$ is a standard Dirac sequence, multiply the resulting equation by $\vr^{s}(x,y,\xi,t)=\vr^{s,0}(x-y+f(\xi)z(t))$ and integrate in $x$ and $\xi$ to get
\begin{align*}
 & \sum_{i=1}^{N}\int\vr^{v}(u(x,t)-\eta)\vr_{\d}(u(x,t)-\xi)\partial_{x_{i}}\left(\int\s_{ik}(\tau)\chi(x,\tau,t)d\tau\right)\vr^{s}(x,y,\xi,t)dxd\xi\\
 & =\sum_{i=1}^{N}\int\partial_{x_{i}}\left(\int\s_{ik}(\tau)\vr^{v}(\tau-\eta)\vr_{\d}(\tau-\xi)\chi(x,\tau,t)d\tau\right)\vr^{s}(x,y,\xi,t)dxd\xi\\
 & =-\sum_{i=1}^{N}\int\left(\int\s_{ik}(\tau)\vr^{v}(\tau-\eta)\vr_{\d}(\tau-\xi)\chi(x,\tau,t)d\tau\right)\partial_{x_{i}}\vr^{s}(x,y,\xi,t)dxd\xi.
\end{align*}
Letting  $\d\to0$ we obtain 
\begin{align*}
 & \sum_{i=1}^{N}\int\vr^{v}(u(x,t)-\eta)\partial_{x_{i}}\left(\int\s_{ik}(\tau)\chi(x,\tau,t)d\tau\right)\vr^{s}(x,y,u(x,t),t)dx\\
 & =-\sum_{i=1}^{N}\int \s_{ik}(\xi)\vr^{v}(\xi-\eta)\chi(x,\xi,t)\partial_{x_{i}}\vr^{s}(x,y,\xi,t)dxd\xi,
\end{align*}
and, thus,
\begin{align*}
 & \sum_{i=1}^{N}\int\partial_{x_{i}}\b_{ik}(u(x,t))\vr(x,y,u(x,t),\eta,t)dx=-\sum_{i=1}^{N}\int\int\s_{ik}(\xi)\chi(x,\xi,t)\partial_{x_{i}}\vr(x,y,\xi,\eta,t)dxd\xi.
\end{align*}

Let now $u$ be a pathwise entropy solution to \eqref{eq:scl}. It follows from \cite{CP03} that there exists a kinetic solution $\td u$ to \eqref{eq:scl}, which, in view of the above, is also a pathwise entropy solution. Since, by  Theorem~\ref{thm:unique} below, pathwise entropy solutions are unique, it follows that $u=\td u$ and hence $u$ is a kinetic solution. 
\end{proof}

\subsection*{The well-posedness of pathwise entropy solutions} The first result here is the uniqueness and stability of pathwise entropy solutions with respect to the initial condition and driving paths.

\begin{thm}\label{thm:unique} Assume \eqref{takis data1} and \eqref{takis data2} and let $u^{(1)},u^{(2)}\in L^{\infty}([0,T];BV(\R^{N}))$ be two pathwise entropy solutions to \eqref{eq:scl} with driving signals $z^{(1)},z^{(2)}\in C_{0}([0,T];\R^{N})$ and initial values $u_{0}^{1},u_{0}^{2}\in(L^{\infty} \cap BV)(\R^{N})$. Then, for all $s\le t\in[0,T]$, there exists $C>0$, which may depend on $\|u^{(1)}\|_{L^{\infty}([s,t];(L^{\infty}\cap BV)(\R^{N}))}$ and  $\|u^{(2)}\|_{L^{\infty}([s,t];(L^{\infty}\cap BV)(\R^{N}))}$,
such that  
\begin{equation}\label{eq:gen_kinetic_unique-2}
\left\{\begin{aligned}
\|u^{(1)}(t)-u^{(2)}(t)\|_{L^{1}(\R^{N})}\le&  \|u^{(1)}(s)-u^{(2)}(s)\|_{L^{1}(\R^{N})}+C \|z^{(1)}-z^{(2)}\|_{C([s,t];\R^{N})}^{1/2}
 \\&+C\|z^{(1)}-z^{(2)}\|_{C([s,t];\R^{N})}.
\end{aligned}\right.
\end{equation}
\end{thm}

The next result is about the  existence of pathwise entropy solutions, which, in view of Theorem~\ref{thm:unique}, are unique.

\begin{thm}\label{thm:existence} (i) Assume \eqref{takis data1} and \eqref{takis data2} and let $u_{0}\in(L^{\infty}\cap BV)(\R^{N})$ and  $z\in C_{0}([0,T];\R^{N})$. Then there exists a pathwise entropy solution $u$ to \eqref{eq:scl} satisfying, for all $p\in[1,\infty]$, 
\begin{align}
  \|u\|_{L^{\infty}([0,T];L^{p}(\R^{N}))}\le\|u_{0}\|_{p},\label{energy est1}\\
  \|u\|_{L^{\infty}([0,T];BV(\R^{N}))}  \le\|u_{0}\|_{BV},  \label{energy est2}
\end{align}
and
\begin{align}
\int_{\R^{N}\times\R\times\R_{+}}m(x,\xi,t)dxd\xi dt+\int_{\R^{N}\times\R_{+}}\sum_{k=1}^{N}\left(\sum_{i=1}^{N}\partial_{x_{i}}\b_{ik}(u)\right)^{2} & dxdt\le\frac{1}{2}\|u_{0}\|_{2}^{2}.\label{eq:measure_par_est-1}
\end{align}
Moreover, for $p\in(-1,\infty)$ and $t\geq0$, \textbf{} 
\begin{align}
\|u(\cdot,t)\|_{p+2}^{p+2} & +(p+2)(p+1)\int_{0}^{t}\int_{{\R^{N}}\times\R}|\xi|^{p}d(m+n)(x,\xi,s)=\|u_{0}\|_{p+2}^{p+2}.\label{eq:measure_bound}
\end{align}
(ii)~Assume \eqref{takis data1} and \eqref{takis data2} and let $u_{0}\in(L^{\infty}\cap L^1)(\R^{N})$ and $z\in C_{0}([0,T];\R^{N})$. Then there exists a pathwise entropy solution $u$ to \eqref{eq:scl} satisfying \eqref{energy est1} for all $p\in[1,\infty]$,  \eqref{eq:measure_par_est-1} and \eqref{eq:measure_bound} for $p\in(-1,\infty)$ and $t\geq0$.

\end{thm}

It follows easily from the arguments leading to Theorem \ref{thm:existence} that the well-posedness theory for \eqref{eq:scl} extends to \eqref{eq:scl} set on the torus $\TT^{N}$. 

\begin{prop}
 Assume \eqref{takis data1} and \eqref{takis data2} and let $u_{0}\in(L^{\infty}\cap BV)(\TT^{N})$ and  $z\in C_{0}([0,T];\R^{N})$. Then, there exists a unique pathwise entropy solution $u\in C\big([0,\infty);L^{1}(\TT^{N})\big)\cap L^{\infty}([0,T];(L^{\infty}\cap BV)(\TT^{N}))$, for all $T>0$, and the solution operator is an $L^{1}$-contraction and, hence, is defined on $L^{1}(\TT^{N}).$
\end{prop}

\subsection*{The long time behavior and regularity}\label{sub:Long-time-behavior-and}
Here we assume that $z$ is a Brownian motion. We start by considering the long time behavior of \eqref{eq:scl-torus}. The claim is  that, under a genuine nonlinearity assumption for $F$ and $A$ (see \eqref{flux} below), as $t\to\infty$ and almost surely (a.s. for short) in $\o$, 
\[
u(\cdot,t;\omega,u_{0})\to\bar{u}_{0}:=\int_{\TT^{N}}u_{0}(x)dx.
\]
Moreover, we provide an estimate for the rate of convergence, we show that $\d_{\bar{u}_{0}}$ is the unique, strongly mixing invariant measure of the associated Markov semigroup and that $\bar{u}_{0}$ is the random attractor of the associated random dynamical system (see Theorem~\ref{thm:ltb} below).

\noindent Without loss of generality we consider the filtered probability space $(\O,\mcF,(\mcF_{t})_{t\in\R_{+}},\P)$ with the canonical realization of the two-sided Brownian motion on $\O=C_{0}(\R;\R^{N}):=\{b\in C(\R;\R^{N})\ \text{ and }\ b(0)=0\}$, and $\P$, $\E$, $\mcF_{t}$ and $\bar{\mcF}_{t}$ respectively the two-sided standard Gaussian measure on $\O$, the expectation with respect to $\P$, the canonical, uncompleted filtration and its completion. 

A measurable map $m:\O\to\mcM$, where $\mcM$  is the space of nonnegative bounded measures on $\TT^{N}\times\R\times[0,T]$, is a kinetic measure, if the process $t\mapsto m(\text{·},[0,t])$ with values in the space of nonnegative bounded measures on $\TT^{N}\times\R$ is $\mcF_{t}$-adapted.

A map $u:\TT^{N}\times[0,\infty)\times\O\to\R$ is called a (stochastic) pathwise entropy solution to \eqref{eq:scl-torus} if, for all $\o\in\O$, $u(\cdot,\o)$ is a pathwise entropy solution to \eqref{eq:scl-torus}, $n,m$ are kinetic measures and $t\mapsto u(\cdot, t)$ is an $\mcF_{t}$-adapted process in $L^{1}(\TT^{N})$.

Since the entropy solution to \eqref{eq:scl-torus} is constructed in a pathwise manner, for each $u_{0}\in L^{1}(\TT^{N})$ and $t\geq0$, the map 
\begin{equation}
\vp(t,\o)u_{0}:=u(\cdot,t;\o,u_{0})\label{takis0}
\end{equation}
defines a continuous random dynamical system (RDS) on $L^{1}(\TT^{N})$. For some background on RDS we refer to \cite[Appendix A]{GS14-2}.  

\noindent The associated Markovian semigroup $(P_{t})_{t\geq0}$ is given, for each bounded measurable function $f$ on $L^{1}(\TT^{N})$, $u_{0}\in L^{1}(\TT^{N})$ and $t\geq0$, by 
\[
P_{t}f(u_{0}):=\E f(u(\cdot,t;\cdot,u_{0}))=\E f(\vp(t,\cdot)u_{0}).
\]
As usual, by duality we may consider the action of $(P_{t})_{t\geq0}$ on the space $\mcM_{1}$ of probability measures on $L^{1}(\TT^{N})$, that is, for each $\mu\in\mcM_{1}$, we define 
\[
P_{t}^{*}\mu(f):=\int_{L^{1}}P_{t}f(x)d\mu(x).
\]
A probability measure $\mu$ is an invariant measure for $(P_{t})_{t\geq0}$ if,  for all $t\ge0$, 
\[
P_{t}^{*}\mu=\mu ,
\]
and $\mu$ is said to be strongly mixing if, for each $\nu\in\mcM_{1}$ and as $t\to\infty$, 
\[
P_{t}^{*}\nu\rightharpoonup\mu\ \text{weak}\star\text{ in }\mcM_{1}.
\]
\noindent The study of the long time behavior of the stochastic entropy solutions is based on a new stochastic averaging-type lemma (Theorem~\ref{thm:reg} below) which holds true under the following genuine nonlinearity condition:
\begin{equation}
\begin{cases}
\text{there exist \ensuremath{\t\in(0,1]}\,\ and \ensuremath{C>0}\,\ such that, for all \ensuremath{\s\in S^{N-1}}, \ensuremath{z\in\R^{N}}and \ensuremath{\ve>0},}\\[1mm]
\hskip1in|\{\xi\in\R:|f(\xi)\sigma-z|^{2}+\sigma A(\xi)\sigma\le\ve\}|\le C\ve^{\frac{\theta}{2}},
\end{cases}\label{flux}
\end{equation}
where the product of the two vectors $f(\xi)\sigma$ is defined by $\left(f(\xi)\sigma\right)^{i}:=f^{i}(\xi)\sigma^{i}$ and  $\sigma A(\xi)\sigma:=\sum_{i,j}a_{ij}(\xi)\s_{i}\s_{j}$. We choose the scaling $\frac{\t}{2}$ on the right hand side in order to be consistent with our previous work \cite{GS14-2}. 

\noindent Also note that in the corresponding deterministic case, that is \eqref{eq:scl-torus} with $z(t)=(t,\dots,t)$, the expectation is that the non-degeneracy condition should be (see Tadmor and Tao \cite{TT07}) 
\begin{equation}
|\{\xi\in\R:|f(\xi)\cdot\sigma-z|+\sigma A(\xi)\sigma\le\ve\}|\lesssim\ve^{\frac{\theta}{2}}.\label{eq:det_scaling}
\end{equation}
In \eqref{eq:det_scaling} the hyperbolic and parabolic parts $|f(\xi)\cdot\sigma-z|$ and $\sigma A(\xi)\sigma$ scale respectively linearly and quadratically in $\s$. In contrast, in the stochastic case both the hyperbolic and parabolic parts scale quadratically in $\s$. This change in the scaling of the hyperbolic part in \eqref{flux} is due to the quadratic scaling of Brownian motion. The fact that both the hyperbolic and the parabolic part of \eqref{eq:scl-torus} are scaled quadratically is crucial for the proof of the averaging lemma and allows a significantly simpler treatment than in \cite{TT07}. 

For notational simplicity we set 
  $$p_{0}:=0\vee p_{1}\vee p_{2}.$$

\noindent \begin{thm}\label{thm:ltb} Assume \eqref{takis data3}, \eqref{b}, \eqref{id} and \eqref{flux}. Then, as $t\to\infty$ 
\begin{equation}
u(\cdot,t;\cdot,u_{0})\to\bar{u}_{0}\ \text{ in \ensuremath{L^{1}(\O;L^{1}(\TT^{N}))}\,\ and \ensuremath{\P}-a.s. in \ensuremath{L^{1}(\TT^{N})}.}\label{eq:main_conv}
\end{equation}
Moreover, for $t\geq1$ and $u_{0}\in L^{2+p_{0}}(\TT^{N})$, 
\[
\E\|u(\cdot,t;\cdot,u_{0})-\bar{u}_{0}\|_{1}\le t^{-\frac{\t}{4+\t}}\left(\|u_{0}\|_{2+p_{0}}^{2+p_{0}}+1\right).
\]
In particular, $\d_{\bar{u}_{0}}$ is the unique invariant measure for $(P_{t})_{t\geq0}$ on $L_{\bar{u}_{0}}^{1}(\TT^{N})$ and is strongly mixing. Restricted to $L_{\bar{u}_{0}}^{1}(\TT^{N})$ the random dynamical system $\vp$ has ${\bar{u}_{0}}$ as a forward and pullback random attractor. \end{thm} 

\noindent Note that, if $u_{0}\in L^{\infty}(\TT^{N}),$ then the convergence in \eqref{eq:main_conv} holds in $L^{p}(\O;L^{p}(\TT^{N}))$ and $\P$-a.s.\ in $L^{p}(\TT^{N})$  for all $p\in[1,\infty)$, since $\|u(\cdot,t;\cdot,u_{0})\|_{\infty}\le\|u_{0}\|_{\infty}$ (see Theorem \ref{thm:existence}).

\no Using essentially the same estimates leading to Theorem \ref{thm:ltb} we also prove a new regularity result, which extends the estimates obtained in \cite{LPS12,GS14-2} for the hyperbolic case, that is, when $A\equiv0$ and in \cite{TT07} for the deterministic problem. 

\begin{thm}\label{thm:reg} Assume \eqref{takis data3}, \eqref{b}, \eqref{id},  \eqref{flux} and let $u_{0}\in L^{2+p_{0}}(\TT^{N})$. Then, for all $\l\in(0,\frac{2\t}{\t+2})$, $\d>0$ and $T>0$, 
\[
\E\int_{0}^{T}\|u(t)\|_{W^{\l,1}}dt\lesssim(1+\|u_{0}\|_{2+p_{0}}^{2+p_{0}})\ \text{ and }\ \sup_{t\geq\d}\E\|u(t)\|_{W^{\l,1}}<\infty.
\]
\end{thm}

\section{The stability of pathwise entropy solutions\label{sec:unique}}

We present the proof of Theorem \ref{thm:unique}. Since it is rather long, we have divided it into several parts and we have chosen to present some technical arguments in the Appendix \ref{sec:err-est}. To help the reader, in the first part of this section we present an informal overview of the arguments and the main ideas and then we continue with the actual proof.

\subsection*{Informal arguments and main ideas}
We present  an informal overview of the proof of uniqueness of pathwise entropy solutions. While several of the arguments presented next need further justification, all the main steps of the rigorous proof are included. In order to make everything rigorous, it is necessary to introduce appropriate smoothing/mollifications and doubling of variables. All these  create significant technical difficulties since it is necessary to estimate the errors etc..

\textbf{\textit{Step 1: }}\textit{Transformation.}  The definition of a pathwise entropy solution corresponds to the transformation
\[
\td\chi(x,\xi,t):=\chi(x+f(\xi)z_{t},\xi,t),
\]
with $\td\chi$ solving, informally,
\begin{align}
\partial_{t}\td\chi(x,\xi,t) & =\partial_{t}(\chi(x+f(\xi)z_{t},\xi,t))\nonumber \\
 & =D_{x}\chi(x+f(\xi)z_{t},\xi,t)\cdot f(\xi)\dot{z}_{t}+\partial_{t}\chi(x+f(\xi)z_{t},\xi,t)\label{eq:informal_transform}\\
 & =\sum_{i,j}a_{ij}(\xi)\partial_{ij}^{2}\td\chi(x,\xi,t)+(\partial_{\xi}q)(x+f(\xi)z_{t},\xi,t).\nonumber 
\end{align}
As noted earlier, the point of this transformation is that the derivative of $z$ does not appear in \eqref{eq:informal_transform}. As compared to the usual proof of the stability  of kinetic solutions in the deterministic case (see, for example, \cite{P02}) additional complications appear since the right hand side is not anymore the derivative of a measure. 

Indeed 
\begin{equation}
(\partial_{\xi}q)(x+f(\xi)z_{t},\xi,t)=\partial_{\xi}(q(x+f(\xi)z_{t},\xi,t))-D_{x}q(x+f(\xi)z_{t},\xi,t)\cdot\dot{f}(\xi)z_{t}\label{eq:non-measure}
\end{equation}
and the second term causes additional difficulties. 

Making use of the special properties of the function $\chi$, we find  that 
\begin{equation}\label{eq:rewritesquare}
\begin{cases}
\|u^{(1)}(t)-u^{(2)}(t)\|_{L^{1}(\R^{N})}  =\int(\chi^{(1)}(x,\xi,t)-\chi^{(2)}(x,\xi,t))^{2}dxd\xi \\[1.5mm]
 =\int|(\chi^{(1)}(x,\xi,t)|+|\chi^{(2)}(x,\xi,t)|-2\chi^{(1)}(x,\xi,t)\chi^{(2)}(x,\xi,t)dxd\xi.
\end{cases}
\end{equation}
The first step is to rewrite \eqref{eq:rewritesquare} in terms of $\td\chi$. Obviously,
\begin{align*}
\int|\chi^{(1)}(x,\xi,t)|+|\chi^{(2)}(x,\xi,t)|dxd\xi & =\int|\td\chi^{(1)}(x,\xi,t)|+|\td\chi^{(2)}(x,\xi,t)|dxd\xi.
\end{align*}
For the third term in \eqref{eq:rewritesquare}, it is easy to see that 
\begin{align*}
 & \left|\int\chi^{(1)}(x,\xi,t)\chi^{(2)}(x,\xi,t)dxd\xi-\int\td\chi^{(1)}(x,\xi,t)\td\chi^{(2)}(x,\xi,t)dxd\xi\right|\\
 & \lesssim\|u^{(2)}\|_{L^{\infty}([0,T];(L^{\infty}\cap BV)(\R^{N}))}|z_{t}^{(1)}-z_{t}^{(2)}|,
\end{align*}
and, hence, for some $C>0$,
\begin{align*}
\|u^{(1)}(\cdot)-u^{(2)}(\cdot)\|_{L^{1}(\R^{N})}\Big|_{s}^{t}\le & \int(\td\chi^{(1)}(x,\xi,\cdot)-\td\chi^{(2)}(x,\xi,\cdot))^{2}dxd\xi\Big|_{s}^{t}\\[1mm]
 & +C\|u^{(2)}\|_{L^{\infty}([0,T];(L^{\infty}\cap BV)(\R^{N}))}\|z^{(1)}-z^{(2)}\|_{C([s,t];\R^{N})}.
\end{align*}
Therefore, it is enough to estimate
\begin{align*}
 & \partial_{t}\int(\td\chi^{(1)}(x,\xi,t)-\td\chi^{(2)}(x,\xi,t))^{2}dxd\xi\\
 & =\partial_{t}\int\big(|\td\chi^{(1)}(x,\xi,t)|+|\td\chi^{(2)}(x,\xi,t)|-2\td\chi^{(1)}(x,\xi,t)\td\chi^{(2)}(x,\xi,t)\big)dxd\xi,
\end{align*}
the advantage being that here we may use \eqref{eq:informal_transform} which makes sense also for non-smooth $z$. 

\textbf{\textit{Step 2:}} \textit{The product term.} Next we note that
\begin{align}
 & -2\frac{d}{dt}\int\td\chi^{(1)}(x,\xi,t)\td\chi^{(2)}(x,\xi,t)dxd\xi\nonumber \\
 & =-2\int[\td\chi^{(1)}(x,\xi,t)\partial_{t}\td\chi^{(2)}(x,\xi,t)+\partial_{t}\td\chi^{(1)}(x,\xi,t)\td\chi^{(2)}(x,\xi,t)]dxd\xi\label{eq:dt_product}\\
 & =-2\sum_{i,j}\int[\td\chi^{(1)}(x,\xi,t)a_{ij}(\xi)\partial_{ij}^{2}\td\chi^{(2)}(x,\xi,t)+a_{ij}(\xi)\partial_{ij}^{2}\td\chi^{(1)}(x,\xi,t)\td\chi^{(2)}(x,\xi,t)]dxd\xi\nonumber \\
 & +2\int[\td\chi^{(1)}(x,\xi,t)(\partial_{\xi}q^{(2)})(x+f(\xi)z_{t}^{(2)},\xi,t)+(\partial_{\xi}q^{(1)})(x+f(\xi)z_{t}^{(1)},\xi,t)\td\chi^{(2)}(x,\xi,t)]dxd\xi\nonumber 
\end{align}
and
\begin{align*}
 & 2\int[\td\chi^{(1)}(x,\xi,t)(\partial_{\xi}q^{(2)})(x+f(\xi)z_{t}^{(2)},\xi,t)+(\partial_{\xi}q^{(1)})(x+f(\xi)z_{t}^{(1)},\xi,t)\td\chi^{(2)}(x,\xi,t)]dxd\xi\\
 & =2\int[\chi^{(1)}(x+f(\xi)(z_{t}^{(1)}-z_{t}^{(2)}),\xi,t)\partial_{\xi}q^{(2)}(x,\xi,t)+\partial_{\xi}q^{(1)}(x,\xi,t)\chi^{(2)}(x+f(\xi)(z_{t}^{(2)}-z_{t}^{(1)}),\xi,t)]dxd\xi\\
 & =2\int[(\partial_{\xi}\chi^{(1)})(x+f(\xi)(z_{t}^{(1)}-z_{t}^{(2)}),\xi,t)q^{(2)}(x,\xi,t)+q^{(1)}(x,\xi,t)(\partial_{\xi}\chi^{(2)})(x+f(\xi)(z_{t}^{(2)}-z_{t}^{(1)}),\xi,t)]dxd\xi\\
 & +Err^{(1,2)}.
\end{align*}
The error term $Err^{(1,2)}$, which is a consequence of  the defect identified in \eqref{eq:non-measure}, is given by
\begin{align*}
Err^{(1,2)}(t):=2\int \bigg( {f'}(\xi)(z_{t}^{(1)}-z_{t}^{(2)}) & \big(D_{x}\chi^{(1)}(x+f(\xi)(z_{t}^{(1)}-z_{t}^{(2)}),\xi,t)q^{(2)}(x,\xi,t)\\
 & -q^{(1)}(x,\xi,t)D_{x}\chi^{(2)}(x+f(\xi)(z_{t}^{(2)}-z_{t}^{(1)}),\xi,t)\big)\bigg)dxd\xi.
\end{align*}
We set 
$$\td q^{(i)}(x,\xi,t):=q^{(i)}(x+f(\xi)z_{t},\xi,t) \ \text{ and} \  \td u^{(i)}(x,\xi,t):=u^{(i)}(x+f(\xi)z_{t}^{(i)},t),$$ 
and observe that 
\begin{align*}
 & 2\int\Big((\partial_{\xi}\chi^{(1)})(x+f(\xi)(z_{t}^{(1)}-z_{t}^{(2)}),\xi,t)q^{(2)}(x,\xi,t)+q^{(1)}(x,\xi,t)(\partial_{\xi}\chi^{(2)})(x+f(\xi)(z_{t}^{(2)}-z_{t}^{(1)}),\xi,t)\Big)dxd\xi\\
 & =2\int\Big((\partial_{\xi}\chi^{(1)})(x+f(\xi)z_{t}^{(1)},\xi,t)\td q^{(2)}(x,\xi,t)+\td q^{(1)}(x,\xi,t)(\partial_{\xi}\chi^{(2)})(x+f(\xi)z_{t}^{(2)},\xi,t)\Big)dxd\xi\\
 & =2\int\Big((\d(\xi)-\d(\xi-\td u^{(1)}(x,\xi,t)))\td q^{(2)}(x,\xi,t)+\td q^{(1)}(x,\xi,t)(\d(\xi)-\d(\xi-\td u^{(2)}(x,\xi,t)))\Big)dxd\xi\\
 & \le-2\int\Big(\d(\xi-\td u^{(1)}(x,\xi,t))\td n^{(2)}(x,\xi,t)+\td n^{(1)}(x,\xi,t)\d(\xi-\td u^{(2)}(x,\xi,t)))\Big)dxd\xi\\
 & +2\int (q^{(2)}(x,0,t)+q^{(1)}(x,0,t))dx.
\end{align*}
The key difference  with  the purely hyperbolic case is that here we do not drop the terms containing the parabolic dissipation measures $\td n^{(1)} \ \text{ and} \ \td n^{(2)}$. Instead, we will   use them to compensate the additional parabolic terms appearing in \eqref{eq:dt_product}. 

In conclusion, we have 
\begin{align*}
 & -2\frac{d}{dt}\int\td\chi^{(1)}(x,\xi,t)\td\chi^{(2)}(x,\xi,t)dxd\xi\le I^{par}(t)+I^{hyp}(t)+Err^{(1,2)}(t),
\end{align*}
where
\begin{align*}
I^{par}(t):= & -2\sum_{i,j}\int\big(\td\chi^{(1)}(x,\xi,t)a_{ij}(\xi)\partial_{ij}^{2}\td\chi^{(2)}(x,\xi,t)+a_{ij}(\xi)\partial_{ij}^{2}\td\chi^{(1)}(x,\xi,t)\td\chi^{(2)}(x,\xi,t)\big)dxd\xi\\
 & -2\int\big(\td n^{(1)}(x,\xi,t)\d(\xi-\td u^{(2)}(x,\xi,t))+\d(\xi-\td u^{(1)}(x,\xi,t))\td n^{(2)}(x,\xi,t)\big)dxd\xi
\end{align*}
and
\[
I^{hyp}(t):=2\int\big( q^{(2)}(x,0,t)+q^{(1)}(x,0,t)\big)dx.
\]

\textbf{\textit{Step 3:}}~\textit{The  hyperbolic term $I^{hyp}$.}  

~Multiplying \eqref{eq:informal_transform} by $\sgn(\xi)$ and integrating yields
\begin{align*}
\frac{d}{dt}\int|\td\chi(x,\xi,t)|dxd\xi & =\frac{d}{dt}\int\sgn(\xi)\td\chi(x,\xi,t)dxd\xi\\
 & =\sum_{i,j}\int\sgn(\xi)a_{ij}(\xi)\partial_{ij}^{2}\td\chi(x,\xi,t)dxd\xi+\int\sgn(\xi)(\partial_{\xi}q)(x+f(\xi)z_{t},\xi,t)dxd\xi\\
 & =-2\int q(x,0,t)dx,
\end{align*}
and thus
\[
I^{hyp}=-\frac{d}{dt}\int|\td\chi^{1}|dxd\xi-\frac{d}{dt}\int|\td\chi^{2}|dxd\xi.
\]

\textbf{\textit{Step 4:}}\textit{ The parabolic term $I^{par}$.}
~Using integration by parts, we observe 
\begin{align*}
 & -2\sum_{i,j}\int\big(\td\chi^{(1)}(x,\xi,t)a_{ij}(\xi)\partial_{ij}^{2}\td\chi^{(2)}(x,\xi,t)+a_{ij}(\xi)\partial_{ij}^{2}\td\chi^{(1)}(x,\xi,t)\td\chi^{(2)}(x,\xi,t)\big)dxd\xi\\
 & =4\sum_{i,j}\int a_{ij}(\xi)\partial_{j}\td\chi^{(1)}(x,\xi,t)\partial_{i}\td\chi^{(2)}(x,\xi,t)dxd\xi\\
 & =4\sum_{i,j}\sum_{k}\int\s_{ik}(\xi)\s_{kj}(\xi)\partial_{j}\td\chi^{(1)}(x,\xi,t)\partial_{i}\td\chi^{(2)}(x,\xi,t)dxd\xi
\end{align*}
and informally (this is where the chain-rule will be required in the rigorous proof)
\begin{align*}
 & \sum_{i,j}\int\s_{ik}(\xi)\s_{kj}(\xi)\partial_{j}\td\chi^{(1)}(x,\xi,t)\partial_{i}\td\chi^{(2)}(x,\xi,t)dxd\xi\\
 & =\sum_{i,j}\int\s_{kj}(\xi)\partial_{j}\td u^{(1)}(x,\xi,t)\d(\xi-\td u^{(1)}(x,\xi,t))\s_{ik}(\xi)\partial_{i}\td u^{(2)}(x,\xi,t)\d(\xi-\td u^{(2)}(x,\xi,t))dxd\xi\\
 & =\sum_{i,j}\int\partial_{j}\b_{kj}(\td u^{(1)}(x,\xi,t))\d(\xi-\td u^{(1)}(x,\xi,t))\partial_{i}\b_{ik}(\td u^{(2)}(x,\xi,t))\d(\xi-\td u^{(2)}(x,\xi,t))dxd\xi.
\end{align*}
Hence,
\begin{align*}
 & -2\sum_{i,j}\int\big(\td\chi^{(1)}(x,\xi,t)a_{ij}(\xi)\partial_{ij}^{2}\td\chi^{(2)}(x,\xi,t)+a_{ij}(\xi)\partial_{ij}^{2}\td\chi^{(1)}(x,\xi,t)\td\chi^{(2)}(x,\xi,t)\big)dxd\xi\\
 & =4\sum_{i,j}\sum_{k}\int\partial_{j}\b_{kj}(\td u^{(1)}(x,\xi,t))\d(\xi-\td u^{(1)}(x,\xi,t))\partial_{i}\b_{ik}(\td u^{(2)}(x,\xi,t))\d(\xi-\td u^{(2)}(x,\xi,t))dxd\xi.
\end{align*}

Since
\begin{align*}
 & \left(\sum_{i}\partial_{i}\b_{ik}(\td u^{(1)}(x,t))\right)^{2}+\left(\sum_{j}\partial_{j}\b_{jk}(\td u^{(2)}(x,t))\right)^{2}\ge2\sum_{i,j}\partial_{i}\b_{ik}(\td u^{(1)}(x,t))^{2}\partial_{j}\b_{jk}(\td u^{(2)}(x,t)),
\end{align*}
using \eqref{eq:diss_meas} we find
\begin{align*}
 & -\int\td n^{(1)}(x,\xi,t)\d(\xi-\td u^{(2)}(x,t))+\d(\xi-\td u^{(1)}(x,t))\td n^{(2)}(x,\xi,t)dxd\xi\\
 & \le-2\sum_{k}\sum_{i,j}\int\d(\xi-\td u^{(1)}(x,\xi,t))\d(\xi-\td u^{(2)}(x,\xi,t))\partial_{i}\b_{ik}(\td u^{(1)}(x,\xi,t))\partial_{j}\b_{jk}(\td u^{(2)}(x,\xi,t))dxd\xi,
\end{align*}
and, in conclusion,
\[
I^{par}\le0.
\]

\textbf{\textit{Step 5:}}
Combining the previous steps we find
\begin{align*}
 & \frac{d}{dt}\int[|\td\chi^{1}|-2\td\chi^{(1)}(x,\xi,t)\td\chi^{(2)}(x,\xi,t)+|\td\chi^{2}|]dxd\xi\le Err^{(1,2)}(t).
\end{align*}
Thus, it remains to estimate $Err^{(1,2)}$. For this we  approximate $\chi^{(i)}$ by convolution with parameter $\ve$, that is, roughly speaking, we replace $\chi^{(i)}$ by a smooth $\chi^{(i),\ve}$. Since $\chi^{(i)}$ has bounded variation,  this introduces an error of order $\ve$, while the gradient $D_{x}\chi^{(i),\ve}$ blows up like $\frac{1}{\ve}$ as  $\ve \to 0$. This leads to the estimate
\[
Err^{(1,2)}(t)\lesssim \frac{|z_{t}^{(1)}-z_{t}^{(2)}|}{\ve} +\ve,
\]
Thus, choosing $\ve=|z_{t}^{(1)}-z_{t}^{(2)}|^{1/2}$, we obtain the bound 
\[
Err^{(1,2)}(t)\le |z_{t}^{(1)}-z_{t}^{(2)}|^{1/2},
\]
which allows to conclude the proof.

\subsection*{The rigorous proof of Theorem~\ref{thm:unique}}
\no We start with the observation 
\begin{align*}
\|u^{(1)}(t)-u^{(2)}(t)\|_{L^{1}(\R^{N})} & =\int(\chi^{(1)}(x,\xi,t)-\chi^{(2)}(x,\xi,t))^{2}dxd\xi\\
 & =\int\big(|\chi^{(1)}(x,\xi,t)|+|\chi^{(2)}(x,\xi,t)|-2\chi^{(1)}(x,\xi,t)\chi^{(2)}(x,\xi,t)\big)dxd\xi,
\end{align*}
where, for $i=1,2$, $\chi^{(i)}$ is related to $u^{(i)}$ by \eqref{eq:char_fctn}. To prove \eqref{eq:gen_kinetic_unique-2} we want to estimate the time derivative of the right hand side of the above equality. In order to be able to use the solution property, following \cite{LPS13,LPS14,GS14}, we replace $\chi^{(1)},\chi^{(2)}$ by $\vr\ast\chi^{(1)}$, $\vr\ast\chi^{(2)}$ for suitable choices of $\vr$. Then it is possible to differentiate with respect to $t$ at the expense of creating several additional terms that need to be estimated. 

\no \textbf{\textit{Step 1:}}\textit{ Approximation.} We perform three approximations. Firstly we consider the convolutions along characteristics in the space variable, secondly we localize the space and velocity variable and thirdly we regularize in the velocity variable. All these are explained next.

Given $t\in[0,T]$ define 
\begin{align*}
G(t) & :=\int(\chi^{(1)}(y,\eta,t)-\chi^{(2)}(y,\eta,t))^{2}dyd\eta,
\end{align*}
\no and let  $\vr_{\ve}^{s},\vr_{\d}^{v}$ be standard smooth mollifiers, that is  approximations of Dirac masses, and for $i=1,2$ we consider 
\[
\vr_{\ve}^{s,(i)}(x,y,\xi,t):=\vr_{\ve}^{s}(x-f(\xi)z^{(i)}(t)-y)
\]
and
\[
\vr_{\ve,\d}^{(i)}(x,y,\xi,\eta,t):=\vr_{\ve}^{s,(i)}(x,y,\xi,t)\vr_{\d}^{v}(\xi-\eta),
\]
that is  $\vr_{\ve}^{s,(i)}$ is the solution to \eqref{eq:transport} with $z=z^{(i)}$ and initial condition $\vr_{\ve}^{s,(i)}(\cdot,y,\xi,0)=\vr_{\ve}^{s}(\cdot-y)$.  Although we can start from the beginning with the $\vr_{\ve,\d}^{(i)}$'s, in order to keep the presentation simpler and, in view of the different role played by $\ve$ and $\d$, we work first with the 
$\vr_{\ve}^{s,(i)}$'s.

We then define 
\begin{align*}
G_{\ve}(t):= & -2\int\left(\chi^{(1)}\ast\vr_{\ve}^{s,(1)}\right)(y,\eta,t)\left(\chi^{(2)}\ast\vr_{\ve}^{s,(2)}\right)(y,\eta,t)dyd\eta\\
 & +\int\sgn(\eta)(\chi^{(1)}\ast\vr_{\ve}^{s,(1)})(y,\eta,t)dyd\eta+\int\sgn(\eta)(\chi^{(2)}\ast\vr_{\ve}^{s,(2)})(y,\eta,t)dyd\eta.
\end{align*}
Lemma \ref{lem:convolution_est-1-1} in the Appendix \ref{sec:err-est} yields
\begin{align*}
|G_{\ve}(t)-G(t)| & \le\left(\ve+\|f(u^{(1)}(t))\|_{L^\infty (\R^N)} |z^{(1)}(t)-z^{(2)}(t)|\right)\big(BV(u^{(1)}(t))+BV(u^{(2)}(t))\big). 
\end{align*}
\no Hence,
\begin{equation}
\begin{split}
G(t)-G(s)\le  G_{\ve}(t)-G_{\ve}(s) +& 
2\left(\ve+\|f(u^{(1)})\|_{L^{\infty}(\R^{N}\times[s,t])}\|z^{(1)}-z^{(2)}\|_{C([s,t];\R^{N})}\right)\\ &(\|u^{(1)}\|_{L^{\infty}([s,t];BV(\R^{N}))}+\|u^{(2)}\|_{L^{\infty}([s,t];BV(\R^{N}))}),
\end{split}
\label{eq:G-err}
\end{equation}
and to conclude it is enough to derive an appropriate bound on $G_{\ve}(t)-G_{\ve}(s) $. 

We show, and this will be the main part of the proof, that 
\begin{equation}
G_{\ve}(t)-G_{\ve}(s)\lesssim\ve^{-1}\|z^{(1)}-z^{(2)}\|_{C([s,t];\R^{N})},\label{eq:main_est_unique}
\end{equation}
and combining \eqref{eq:G-err} and \eqref{eq:main_est_unique} we find 
\begin{align*}
G(t)-G(s)\lesssim \ve^{-1} \|z^{(1)}-z^{(2)}\|_{C([s,t];\R^{N})}
  +&2\left(\ve+\|f(u^{(1)})\|_{L^{\infty}(\R^{N}\times[s,t])}\|z^{(1)}-z^{(2)}\|_{C([s,t];\R^{N})}\right)\\&(\|u^{(1)}\|_{L^{\infty}([s,t];BV(\R^{N}))}+\|u^{(2)}\|_{L^{\infty}([s,t];BV(\R^{N}))}).
\end{align*}
The conclusion follows by choosing $\ve^2={\|z^{(1)}-z^{(2)}\|_{C([s,t];\R^{N})}}$.  It thus remains to prove \eqref{eq:main_est_unique}.

\textbf{\textit{Step 2:}} \textit{The product term.} We note that, while $\chi^{(i)}\ast\vr_{\ve}^{s,(i)}$ is smooth in $y$, in order to apply Definition \ref{def:path_e-soln-2} we still need to mollify with respect to $\eta$, that is to  consider $\chi^{(i)}\ast\vr_{\ve,\d}^{(i)}$ and, in the end, take $\d\to0$. For this mollification to converge, we first need to localize in $y$ and $\eta$, a fast which requires yet another layer of approximation. 

Therefore, for every $\psi\in C_{c}^{\infty}(\R^{N+1})$ with $\|\psi\|_{C(\R^{N})}\le1$, we introduce
\begin{align*}
G_{\ve,\psi,\d}(t):= & -2\int\psi(y,\eta)\left(\chi^{(1)}\ast\vr_{\ve,\d}^{(1)}\right)(y,\eta,t)\left(\chi^{(2)}\ast\vr_{\ve,\d}^{(2)}\right)(y,\eta,t)dyd\eta\\
 & +\int\psi(y,\eta)(\sgn\ast\vr_{\d}^{v})(\eta)(\chi^{(1)}\ast\vr_{\ve,\d}^{(1)})(y,\eta,t)dyd\eta+\int\psi(y,\eta)(\sgn\ast\vr_{\d}^{v})(\eta)(\chi^{(2)}\ast\vr_{\ve,\d}^{(2)})(y,\eta,t)dyd\eta.
\end{align*}
To simplify the notation we suppress the $\ve,\psi,\d$-dependence in the next few steps. We have: 
\begin{align*}
 & -2\frac{d}{dt}\int\psi(y,\eta)\left(\chi^{(1)}\ast\vr^{(1)}\right)\left(\chi^{(2)}\ast\vr^{(2)}\right)dyd\eta\\
 & =-2\int[\psi(y,\eta)\left(\chi^{(1)}\ast\vr^{(1)}\right)\frac{d}{dt}\left(\chi^{(2)}\ast\vr^{(2)}\right)+\psi(y,\eta)\left(\chi^{(2)}\ast\vr^{(2)}\right)\frac{d}{dt}\left(\chi^{(1)}\ast\vr^{(1)}\right)]dyd\eta\\
 & =-2\sum_{i,j=1}^{N}\int[\psi(y,\eta)\chi^{(1)}(x,\xi,t)\vr^{(1)}(x,y,\xi,\eta,t)\chi^{(2)}(x',\xi',t)a_{ij}(\xi')\partial_{x_{i}x_{j}}^{2}\vr^{(2)}(x',y,\xi',\eta,t)\\
 & +\psi(y,\eta)\chi^{(2)}(x',\xi',t)\vr^{(2)}(x',y,\xi',\eta,t)\chi^{(1)}(x,\xi,t)a_{ij}(\xi)\partial_{x_{i}x_{j}}^{2}\vr^{(1)}(x,y,\xi,\eta,t)]dxd\xi dx'd\xi'dyd\eta\\
 & -2\int[\psi(y,\eta)\chi^{(1)}(x,\xi,t)\vr^{(1)}(x,y,\xi,\eta,t)\vr^{(2)}(x',y,\xi',\eta,t)\partial_{\xi'}q^{(2)}(x',\xi',t)\\
 & +\psi(y,\eta)\chi^{(2)}(x',\xi',t)\vr^{(2)}(x',y,\xi',\eta,t)\vr^{(1)}(x,y,\xi,\eta,t)\partial_{\xi}q^{(1)}(x,\xi,t)]dxd\xi dx'd\xi'dyd\eta.
\end{align*}
We first consider the last two terms on the right hand side. As in the hyperbolic case we note
\begin{align*}
 & -2\int\big(\psi(y,\eta)\chi^{(1)}(x,\xi,t)\vr^{(1)}(x,y,\xi,\eta,t)\vr^{(2)}(x',y,\xi',\eta,t)\partial_{\xi'}q^{(2)}(x',\xi',t)\\
 & +\psi(y,\eta)\chi^{(2)}(x',\xi',t)\vr^{(2)}(x',y,\xi',\eta,t)\vr^{(1)}(x,y,\xi,\eta,t)\partial_{\xi}q^{(1)}(x,\xi,t)\big)dxd\xi dx'd\xi'dyd\eta\\
 & =-2\int\big(\psi(y,\eta)\chi^{(1)}(x,\xi,t)\partial_{\xi}\vr^{(1)}(x,y,\xi,\eta,t)\vr^{(2)}(x',y,\xi',\eta,t)q^{(2)}(x',\xi',t)\\
 & +\psi(y,\eta)\chi^{(2)}(x',\xi',t)\partial_{\xi'}\vr^{(2)}(x',y,\xi',\eta,t)\vr^{(1)}(x,y,\xi,\eta,t)q^{(1)}(x,\xi,t)\big)dxd\xi dx'd\xi'dyd\eta\\
 & +Err^{(1,2)}(t),
\end{align*}
where 
\begin{equation}
\begin{cases}
Err^{(1,2)}(t):=2\int\psi(y,\eta)\left(\chi^{(1)}(x,\xi,t)q^{(2)}(x',\xi',t)+\chi^{(2)}(x',\xi',t)q^{(1)}(x,\xi,t)\right)\\[1.5mm]
(\vr^{(1)}(x,y,\xi,\eta,t)\partial_{\xi'}\vr^{(2)}(x',y,\xi',\eta,t)+\partial_{\xi}\vr^{(1)}(x,y,\xi,\eta,t)\vr^{(2)}(x',y,\xi',\eta,t))dyd\eta dx'd\xi'dxd\xi.
\end{cases}\label{eq:err-1-1}
\end{equation}
We next use the nonnegativity of the Dirac masses $\delta$ to observe that 
\begin{align}
 & -2\int[\psi(y,\eta)\chi^{(1)}(x,\xi,t)\partial_{\xi}\vr^{(1)}(x,y,\xi,\eta,t)\vr^{(2)}(x',y,\xi',\eta,t)q^{(2)}(x',\xi',t)\nonumber \\
 & +\psi(y,\eta)\chi^{(2)}(x',\xi',t)\partial_{\xi'}\vr^{(2)}(x',y,\xi',\eta,t)\vr^{(1)}(x,y,\xi,\eta,t)q^{(1)}(x,\xi,t)]dxd\xi dx'd\xi'dyd\eta\nonumber \\
 & =2\int[\psi(y,\eta)\left(\d(\xi)-\d(\xi-u^{(1)}(x,t))\right)\vr^{(1)}(x,y,\xi,\eta,t)\vr^{(2)}(x',y,\xi',\eta,t)q^{(2)}(x',\xi',t)\nonumber \\
 & +\psi(y,\eta)\left(\d(\xi')-\d(\xi'-u^{(2)}(x',t))\right)\vr^{(2)}(x',y,\xi',\eta,t)\vr^{(1)}(x,y,\xi,\eta,t)q^{(1)}(x,\xi,t)]dxd\xi dx'd\xi'dyd\eta\label{eq:rewrite_hyperbolic_part}\\
 & \le-2\int[\psi(y,\eta)\d(\xi-u^{(1)}(x,t))\vr^{(1)}(x,y,\xi,\eta,t)\vr^{(2)}(x',y,\xi',\eta,t)n^{(2)}(x',\xi',t)\nonumber \\
 & +\psi(y,\eta)\d(\xi'-u^{(2)}(x',t))\vr^{(2)}(x',y,\xi',\eta,t)\vr^{(1)}(x,y,\xi,\eta,t)n^{(1)}(x,\xi,t)]dxd\xi dx'd\xi'dyd\eta\nonumber \\
 & +2\int\psi(y,\eta)\vr^{(1)}(x,y,0,\eta,t)\vr^{(2)}(x',y,\xi',\eta,t)q^{(2)}(x',\xi',t)dx dx'd\xi' dyd\eta\nonumber \\
 & +2\int \psi(y,\eta)\vr^{(1)}(x,y,\xi,\eta,t)\vr^{(2)}(x',y,0,\eta,t)q^{(1)}(x,\xi,t)]dxd\xi dx'dyd\eta,\nonumber 
\end{align}
and, hence,
\begin{align*}
-2\frac{d}{dt}\int\psi(y,\eta)\left(\chi^{(1)}\ast\vr^{(1)}\right)\left(\chi^{(2)}\ast\vr^{(2)}\right)dyd\eta & \le I^{par}(t)+I^{hyp}(t)+Err^{(1,2)}(t),
\end{align*}
with $I^{par}$ and  $I^{hyp}$  denoting respectively the parabolic and hyperbolic terms, that is 
\begin{align*}
I^{par}(t):= & -2\sum_{i,j=1}^{N}\int\big(\psi(y,\eta)\chi^{(1)}(x,\xi,t)\vr^{(1)}(x,y,\xi,\eta,t)\chi^{(2)}(x',\xi',r)a_{ij}(\xi')\partial_{x_{i}'x_{j}'}\vr^{(2)}(x',y,\xi',\eta,s)\\
 & +\psi(y,\eta)\chi^{(2)}(x',\xi',t)\vr^{(2)}(x',y,\xi',\eta,t)\chi^{(1)}(x,\xi,r)a_{ij}(\xi)\partial_{x_{i}x_{j}}^{2}\vr^{(1)}(x,y,\xi,\eta,s)\big)dxd\xi dx'd\xi'dyd\eta\\
 & -2\int\big(\psi(y,\eta)\d(\xi-u^{(1)}(x,t))\vr^{(1)}(x,y,\xi,\eta,t)\vr^{(2)}(x',y,\xi',\eta,t)n^{(2)}(x',\xi',t)\\
 & +\psi(y,\eta)\d(\xi'-u^{(2)}(x',t))\vr^{(2)}(x',y,\xi',\eta,t)\vr^{(1)}(x,y,\xi,\eta,t)n^{(1)}(x,\xi,t)\big)dxd\xi dx'd\xi'dyd\eta,
\end{align*}
and
\begin{align*}
I^{hyp}(t):= & 2\int\psi(y,\eta)\vr^{(1)}(x,y,0,\eta,t)\vr^{(2)}(x',y,\xi',\eta,t)q^{(2)}(x',\xi',t)dx dx'd\xi'dyd\eta\\
 & +2\int \psi(y,\eta)\vr^{(1)}(x,y,\xi,\eta,t)\vr^{(2)}(x',y,0,\eta,t)q^{(1)}(x,\xi,t)dxd\xi dx'dyd\eta;
\end{align*}
note that we call the above expression a hyperbolic term in spite of the occurrence of the parabolic dissipation measures $n^{(1)}$ and $n^{(2)}$, because they have the same structure as and are treated similarly to the hyperbolic case.

\textbf{\textit{Step 3: }}\textit{The hyperbolic terms.} Since both terms in $I^{hyp}$ are  treated similarly, we provide details only for the first one. We have
\begin{align*}
 & 2\int_{s}^{t}\int\psi(y,\eta)\vr^{(1)}(x,y,0,\eta,r)\vr^{(2)}(x',y,\xi',\eta,r)q^{(2)}(x',\xi',r)dxdx'd\xi'dyd\eta dr\\
 & =\int_{s}^{t}\int\psi(y,\eta)\partial_{\xi}\sgn(\xi)\vr^{(1)}(x,y,\xi,\eta,r)\vr^{(2)}(x',y,\xi',\eta,r)q^{(2)}(x',\xi',r)dxd\xi dx'd\xi'dyd\eta dr\\
 & =-\int_{s}^{t}\int\psi(y,\eta)\sgn(\xi)\partial_{\xi}\vr^{(1)}(x,y,\xi,\eta,r)\vr^{(2)}(x',y,\xi',\eta,r)q^{(2)}(x',\xi',r)dxd\xi dx'd\xi'dyd\eta dr\\
 & =\int_{s}^{t}\int\psi(y,\eta)\left(\int\sgn(\xi)\vr^{(1)}(x,y,\xi,\eta,r)dxd\xi\right)\left(\int\partial_{\xi'}\vr^{(2)}(x',y,\xi',\eta,r)q^{(2)}(x',\xi',r)dx'd\xi'\right)dyd\eta dr\\
 & +\int_{s}^{t}Err^{(1)}(r)dr,
\end{align*}
with
\begin{align*}
Err^{(1)}(r):=-\int & \psi(y,\eta)\sgn(\xi)\Big(\partial_{\xi}\vr^{(1)}(x,y,\xi,\eta,r)\vr^{(2)}(x',y,\xi',\eta,r)\\
 & +\vr^{(1)}(x,y,\xi,\eta,r)\partial_{\xi'}\vr^{(2)}(x',y,\xi',\eta,r)\Big)q^{(2)}(x',\xi',r)dxd\xi dx'd\xi'dyd\eta.
\end{align*}
Using the solution property \eqref{eq:kinetic} for $\chi^{(2)}$ and
\begin{align*}
\int\sgn(\xi)\vr^{(1)}(x,y,\xi,\eta,r)dxd\xi & =(\sgn\ast\vr^{v})(\eta)
\end{align*}
we find
\begin{align*}
 & 2\int_{s}^{t}\int\psi(y,\eta)\vr^{(1)}(x,y,0,\eta,r)\vr^{(2)}(x',y,\xi',\eta,r)q^{(2)}(x',\xi',r)dxd\xi'dx'dyd\eta dr\\
 & =\int_{s}^{t}\int\psi(y,\eta)(\sgn\ast\vr^{v})(\eta)\left(\int\partial_{\xi'}\vr^{(2)}(x',y,\xi',\eta,r)q^{(2)}(x',\xi',r)dx'd\xi'\right)dyd\eta dr\\
 & +\int_{s}^{t}Err^{(1)}(r)dr\\
 & =-\int\psi(y,\eta)(\sgn\ast\vr^{v})(\eta)(\chi^{(2)}\ast\vr^{(2)})(y,\eta,\cdot)|_{s}^{t}dyd\eta\\
 & +\int_{s}^{t}\int(\sgn\ast\vr^{v})(\eta)\left(\psi(y,\eta)\sum_{i,j=1}^{N}\int\chi^{(2)}(x',\xi',t)a_{ij}(\xi')\partial_{x_{i}'x_{j}'}\vr^{(2)}(x',y,\xi',\eta,t)dx'd\xi'\right)dyd\eta dr\\
 & +\int_{s}^{t}Err^{(1)}(r)dr.
\end{align*}
Since
\begin{align}\label{eq:x-y-switch}
\partial_{x_{i}}\vr^{(i)}(x,y,\xi,\eta,t)=  -\partial_{y_{i}}\vr^{(i)}(x,y,\xi,\eta,t)\quad\text{and}\quad
\partial_{x_{i}x_{j}}^{2}\vr^{(i)}(x,y,\xi,\eta,t)=  \partial_{y_{i}y_{j}}^{2}\vr^{(i)}(x,y,\xi,\eta,t) 
\end{align}
we have
\begin{align*}
 & \int_{s}^{t}\int(\sgn\ast\vr^{v})(\eta)\psi(y,\eta)\sum_{i,j=1}^{N}\int\chi^{(2)}(x',\xi',t)a_{ij}(\xi')\partial_{x_{i}'x_{j}'}\vr^{(2)}(x',y,\xi',\eta,t)dx'd\xi'dyd\eta dr\\
 & =\int_{s}^{t}\int(\sgn\ast\vr^{v})(\eta)\partial_{y_{i}}\psi(y,\eta)\sum_{i,j=1}^{N}\int\chi^{(2)}(x',\xi',t)a_{ij}(\xi')\partial_{x_{j}'}\vr^{(2)}(x',y,\xi',\eta,t)dx'd\xi'dyd\eta dr\\
 & =:\int_{s}^{t}Err^{loc,(1)}(r)dr.
\end{align*}

In conclusion,
\begin{align*}
\int_{s}^{t}I^{hyp}(r)dr= & -\int\psi(y,\eta)(\sgn\ast\vr^{v})(\eta)(\chi^{(1)}\ast\vr^{(1)})(y,\eta,\cdot)|_{s}^{t}dyd\eta\\
 & -\int\psi(y,\eta)(\sgn\ast\vr^{v})(\eta)(\chi^{(2)}\ast\vr^{(2)})(y,\eta,\cdot)|_{s}^{t}dyd\eta\\
 & +\int_{s}^{t}Err^{(1)}(r)+Err^{(2)}(r)dr+\int_{s}^{t}Err^{loc,(1)}(r)+Err^{loc,(2)}(r)dr,
\end{align*}
with $Err^{(2)}$, $Err^{loc,(2)}$ defined analogously to $Err^{(1)}$, $Err^{loc,(1)}$ respectively.

\textbf{\textit{Step 4: }}\textit{The parabolic terms.} Using \eqref{eq:x-y-switch} we get
\begin{align*}
 & \sum_{i,j=1}^{N}\int\psi(y,\eta)\chi^{(1)}(x,\xi,t)\vr^{(1)}(x,y,\xi,\eta,t)\chi^{(2)}(x',\xi',t)a_{ij}(\xi')\partial_{x'_{i}x'_{j}}\vr^{(2)}(x',y,\xi',\eta,t)dxd\xi dx'd\xi'dyd\eta\\
 & =-\sum_{i,j=1}^{N}\int\psi(y,\eta)\chi^{(1)}(x,\xi,t)\chi^{(2)}(x',\xi',t)a_{ij}(\xi')\partial_{x_{i}}\vr^{(1)}(x,y,\xi,\eta,t)\partial_{x'_{j}}\vr^{(2)}(x',y,\xi',\eta,t)dxd\xi dx'd\xi'dyd\eta\\
 & +Err^{loc,(3)},
\end{align*}
where
\[
Err^{loc,(3)}:=-\sum_{i,j=1}^{N}\int\partial_{y_{i}}\psi(y,\eta)\chi^{(1)}(x,\xi,t)\chi^{(2)}(x',\xi',t)a_{ij}(\xi')\vr^{(1)}(x,y,\xi,\eta,t)\partial_{x'_{j}}\vr^{(2)}(x',y,\xi',\eta,t)dxd\xi dx'd\xi'dyd\eta.
\]
Hence,
\begin{align}
 & -2\sum_{i,j=1}^{N}\int[\psi(y,\eta)\chi^{(1)}(x,\xi,t)\vr^{(1)}(x,y,\xi,\eta,t)\chi^{(2)}(x',\xi',t)a_{ij}(\xi')\partial_{x'_{i}x'_{j}}\vr^{(2)}(x',y,\xi',\eta,t)\nonumber \\
 & +\psi(y,\eta)\chi^{(2)}(x',\xi',t)\vr^{(2)}(x',y,\xi',\eta,t)\chi^{(1)}(x,\xi,t)a_{ij}(\xi)\partial_{x_{i}x_{j}}^{2}\vr^{(1)}(x,y,\xi,\eta,t)]dxd\xi dx'd\xi'dyd\eta\label{eq:par_mol_splitup}\\
 & =4\sum_{i,j=1}^{N}\sum_{k=1}^{N}\int[\psi(y,\eta)\chi^{(1)}(x,\xi,t)\chi^{(2)}(x',\xi',t)\sigma_{ik}(\xi)\s_{kj}(\xi')\nonumber \\
 & \partial_{x_{i}}\vr^{(1)}(x,y,\xi,\eta,t)\partial_{x'_{j}}\vr^{(2)}(x',y,\xi',\eta,t)]dxd\xi dx'd\xi'dyd\eta\nonumber \\
 & +Err^{par}+Err^{loc,(3)}+Err^{loc,(4)},\nonumber 
\end{align}
with
\begin{align*}
Err^{par}:=2\sum_{i,j=1}^{N}\int & \psi(y,\eta)\chi^{(1)}(x,\xi,t)\chi^{(2)}(x',\xi',t)\left(a_{ij}(\xi)-2\sum_{k=1}^{N}\sigma_{ik}(\xi)\s_{kj}(\xi')+a_{ij}(\xi')\right)\\
 & \partial_{x_{i}}\vr^{(1)}(x,y,\xi,\eta,t)\partial_{x'_{j}}\vr^{(2)}(x',y,\xi',\eta,t)dxd\xi dx'd\xi'dyd\eta
\end{align*}
and $Err^{loc,(4)}$ defined analogously to $Err^{loc,(3)}$. 

We now use the chain rule \eqref{eq:chain-rule}  and get
\begin{align*}
 & \sum_{i,j=1}^{N}\int\psi(y,\eta)\chi^{(1)}(x,\xi,t)\chi^{(2)}(x',\xi',t)\sigma_{ik}(\xi)\s_{kj}(\xi')\partial_{x_{i}}\vr^{(1)}(x,y,\xi,\eta,t)\partial_{x'_{j}}\vr^{(2)}(x',y,\xi',\eta,t)dxd\xi dx'd\xi'dyd\eta\\
 & =\int\psi(y,\eta)\left(\sum_{i=1}^{N}\int\chi^{(1)}(x,\xi,t)\sigma_{ik}(\xi)\partial_{x_{i}}\vr^{(1)}(x,y,\xi,\eta,t)dxd\xi\right)\\
 & \left(\sum_{j=1}^{N}\int\chi^{(2)}(x',\xi',t)\s_{kj}(\xi')\partial_{x'_{j}}\vr^{(2)}(x',y,\xi',\eta,t)dx'd\xi'\right)dyd\eta\\
 & =\int\psi(y,\eta)\left(\sum_{i=1}^{N}\int\partial_{x_{i}}\b_{ik}(u^{(1)}(x,t))\vr^{(1)}(x,y,u^{(1)}(x,t),\eta,t)dx\right)\\[1.5mm]
 & \left(\sum_{j=1}^{N}\int\partial_{x'_{j}}\b_{kj}(u^{(2)}(x',t))\vr^{(2)}(x',y,u^{(2)}(x',t),\eta,t)dx'\right)dyd\eta,
\end{align*}
and, therefore,
\begin{align*}
 & -2\sum_{i,j=1}^{N}\int\big(\psi(y,\eta)\chi^{(1)}(x,\xi,t)\vr^{(1)}(x,y,\xi,\eta,t)\chi^{(2)}(x',\xi',t)a_{ij}(\xi')\partial_{x_{i}x_{j}}^{2}\vr^{(2)}(x',y,\xi',\eta,t)\\
 & +\psi(y,\eta)\chi^{(2)}(x',\xi',t)\vr^{(2)}(x',y,\xi',\eta,t)\chi^{(1)}(x,\xi,t)a_{ij}(\xi)\partial_{x_{i}x_{j}}^{2}\vr^{(1)}(x,y,\xi,\eta,t)\big)dxd\xi dx'd\xi'dyd\eta\\
 & =4\sum_{i,j=1}^{N}\sum_{k=1}^{N}\int\psi(y,\eta)\Bigg(\int\partial_{x_{i}}\b_{ik}(u^{(1)}(x,t))\vr^{(1)}(x,y,u^{(1)}(x,t),\eta,t)\\
 & \hskip120pt\partial_{x_{j}}\b_{kj}(u^{(2)}(x',t))\vr^{(2)}(x',y,u^{(2)}(x',t),\eta,t)dxdx'\Bigg)dyd\eta\\
 & +Err^{par}+Err^{loc,(3)}+Err^{loc,(4)}.
\end{align*}
Using \eqref{eq:diss_meas} we analyze next the last two terms in $I^{par}$. We have:
\begin{align*}
 & \int\psi(y,\eta)\d(\xi-u^{(1)}(x,t))\vr^{(1)}(x,y,\xi,\eta,t)\vr^{(2)}(x',y,\xi',\eta,t)n^{(2)}(x',\xi',t)dxd\xi dx'd\xi'dyd\eta\\
 & +\int\psi(y,\eta)\d(\xi'-u^{(2)}(x',t))\vr^{(2)}(x',y,\xi',\eta,t)\vr^{(1)}(x,y,\xi,\eta,t)n^{(1)}(x,\xi,t)dxd\xi dx'd\xi'dyd\eta\\
 & =\int\psi(y,\eta)\vr^{(1)}(x,y,u^{(1)}(x,t),\eta,t)\vr^{(2)}(x',y,\xi',\eta,t)n^{(2)}(x',\xi',t)dxdx'd\xi'dyd\eta\\
 & +\int\psi(y,\eta)\vr^{(2)}(x',y,u^{(2)}(x',t),\eta,t)\vr^{(1)}(x,y,\xi,\eta,t)n^{(1)}(x,\xi,t)dxd\xi dx'dyd\eta\\
 & =\int\psi(y,\eta)\vr^{(1)}(x,y,u^{(1)}(x,t),\eta,t)\vr^{(2)}(x',y,u^{(2)}(t,x'),\eta,t)\sum_{k=1}^{N}\left(\sum_{i=1}^{N}\partial_{x'_{i}}\b_{ik}(u^{(2)}(t,x'))\right)^{2}dxdx'dyd\eta\\
 & +\int\psi(y,\eta)\vr^{(2)}(x',y,u^{(2)}(x',t),\eta,t)\vr^{(1)}(x,y,u^{(1)}(t,x),\eta,t)\sum_{k=1}^{N}\left(\sum_{i=1}^{N}\partial_{x_{i}}\b_{ik}(u^{(1)}(t,x))\right)^{2}dxdx'dyd\eta.
\end{align*}
Since
\begin{align*}
 & \left(\sum_{i=1}^{N}\partial_{x_{i}}\b_{ik}(u^{(1)}(t,x))\right)^{2}+\left(\sum_{i=1}^{N}\partial_{x_{i}}\b_{ik}(u^{(2)}(t,x'))\right)^{2}\ge2\sum_{i,j=1}^{N}\partial_{x_{i}}\b_{ik}(u^{(1)}(t,x))\partial_{x'_{j}}\b_{kj}(u^{(2)}(t,x')),
\end{align*}
we obtain
\begin{align*}
 & -2\int\psi(y,\eta)\d(\xi-u^{(1)}(x,t))\vr^{(1)}(x,y,\xi,\eta,t)\vr^{(2)}(x',y,\xi',\eta,t)n^{(2)}(x',\xi',t)dxd\xi dx'd\xi'dyd\eta\\
 & +\int \psi(y,\eta)\d(\xi'-u^{(2)}(x',t))\vr^{(2)}(x',y,\xi',\eta,t)\vr^{(1)}(x,y,\xi,\eta,t)n^{(1)}(x,\xi,t)dxd\xi dx'd\xi'dyd\eta\\
 & \le-4\int\psi(y,\eta)\vr^{(1)}(x,y,u^{(1)}(x,t),\eta,t)\vr^{(2)}(x',y,u^{(2)}(t,x'),\eta,t)\\
 & \sum_{k=1}^{N}\left(\sum_{i,j=1}^{N}\partial_{x_{i}}\b_{ik}(u^{(1)}(t,x))\partial_{x'_{j}}\b_{jk}(u^{(2)}(t,x'))\right)dxdx'dyd\eta,
\end{align*}
and, then,
\begin{align*}
\int_{s}^{t}I^{par}(r)dr & \le\int_{s}^{t}Err^{par}(r)+Err^{loc,(3)}(r)+Err^{loc,(4)}(r)dr.
\end{align*}

\textbf{\textit{Step 5:}} \textit{The end of the proof.} Combining the estimates in the last two steps we find, now explicitly writing the $\ve,\psi,\d$-dependence,
\begin{align}
 & -2\int\psi(y,\eta)(\chi^{(1)}\ast\vr_{\ve,\d}^{(1)})(y,\eta,\cdot)(\chi^{(2)}\ast\vr_{\ve,\d}^{(2)})(y,\eta,\cdot)dyd\eta\Big|_{s}^{t}\nonumber \\
 & \le-\int\psi(y,\eta)(\sgn\ast\vr_{\d}^{v})(\eta)(\chi^{(1)}\ast\vr_{\ve,\d}^{(1)})(y,\eta,\cdot)|_{s}^{t}dyd\eta\nonumber \\
 & -\int\psi(y,\eta)(\sgn\ast\vr_{\d}^{v})(\eta)(\chi^{(2)}\ast\vr_{\ve,\d}^{(2)})(y,\eta,\cdot)|_{s}^{t}dyd\eta\label{eq:unique_1-2}\\
 & +\int_{s}^{t}\Big(Err_{\ve,\psi,\d}^{(1)}(r)+Err_{\ve,\psi,\d}^{(2)}(r)+Err_{\ve,\psi,\d}^{(1,2)}(r)+Err_{\ve,\psi,\d}^{par}(r)\Big)dr\nonumber \\
 & +\int_{s}^{t}\Big(Err_{\ve,\psi,\d}^{loc,(1)}(r)+Err_{\ve,\psi,\d}^{loc,(2)}(r)+Err_{\ve,\psi,\d}^{loc,(3)}(r)+Err_{\ve,\psi,\d}^{loc,(4)}(r)\Big)dr,\nonumber 
\end{align}
that is 
\begin{align*}
G_{\ve,\psi,\d}(t)-G_{\ve,\psi,\d}(s) & \le\int_{s}^{t}\Big(Err_{\ve,\psi,\d}^{(1)}(r)+Err_{\ve,\psi,\d}^{(2)}(r)+Err_{\ve,\psi,\d}^{(1,2)}(r)+Err_{\ve,\psi,\d}^{par}(r)\Big)dr\\
 & +\int_{s}^{t}\Big(Err_{\ve,\psi,\d}^{loc,(1)}(r)+Err_{\ve,\psi,\d}^{loc,(2)}(r)+Err_{\ve,\psi,\d}^{loc,(3)}(r)+Err_{\ve,\psi,\d}^{loc,(4)}(r)\Big)dr.
\end{align*}
It follows from \eqref{eq:Err_hyp_1}, \eqref{eq:err12_bd}, Lemma~ \ref{lem:par_err} and and Lemma \ref{lem:err-loc-2} in Appendix \ref{sec:err-est} that, for $\psi\in C_{c}^{\infty}(\R^{N+1})$ and $\ve>0$ fixed, 
\begin{align}
 & \lim_{\d\to0}G_{\ve,\psi,\d}(t)-\lim_{\d\to0}G_{\ve,\psi,\d}(s)\nonumber \\
 & \lesssim\int_{s}^{t}\int\|\partial_{\eta}\psi(\cdot,\xi)\|_{C(\R^{N})}(q^{(1)}+q^{(2)})(x,\xi,r)dxd\xi dr+\ve^{-1}\|z^{(1)}-z^{(2)}\|_{C([s,t];\R^{N})}\label{eq:eps-psi-ineq}\\
 & +\int_{s}^{t}\Big(Err_{\ve,\psi}^{loc,(1)}(r)+Err_{\ve,\psi}^{loc,(2)}(r)+Err_{\ve,\psi}^{loc,(3)}(r)+Err_{\ve,\psi}^{loc,(4)}(r)\Big)dr.\nonumber 
\end{align}
Choosing $\psi_{R}\in C_{c}^{\infty}(\R^{N+1})$ such that
\begin{equation}\label{eq:psi_loc}
\psi_{R}(y,\eta)=\begin{cases}
1, & \text{if }|(y,\eta)|\le R\\[1.5mm]
0, & \text{if }|(y,\eta)|>R+1,
\end{cases}\quad\text{and}\quad\|D\psi_{R}\|\le1,
\end{equation}
yields 
\[
\lim_{R\to\infty}\int\|\partial_{\eta}\psi_R(\cdot,\xi)\|_{C(\R^{N})}(q^{(1)}+q^{(2)})(x,\xi,t)dxd\xi=0.
\]
In view of Lemma \ref{lem:err-loc} we also have
\[
\lim_{R\to\infty}\int_{s}^{t}\Big(Err_{\ve,\psi_R}^{loc,(1)}(r)+Err_{\ve,\psi_R}^{loc,(2)}(r)+Err_{\ve,\psi_R}^{loc,(3)}(r)+Err_{\ve,\psi_R}^{loc,(4)}(r)\Big)dr=0
\]
and it is easy to see that 
\[
\lim_{R\to\infty}\lim_{\d\to0}G_{\ve,\psi_{R},\d}(t)=G_{\ve}(t).
\]
Hence, letting  $R\to\infty$ in \eqref{eq:eps-psi-ineq} yields
\begin{align}
 & G_{\ve}(t)-G_{\ve}(s)\le \ve^{-1}\|z^{(1)}-z^{(2)}\|_{C([s,t];\R^{N})},\label{eq:first_est-2-1}
\end{align}
which finishes the proof by step one.

\section{The existence of pathwise entropy solutions\label{sec:existence}}

We prove the existence of pathwise entropy solutions to \eqref{eq:scl} and establish the a priori estimates stated in Theorem \ref{thm:existence}. The solution is found as a limit of solutions of a three-step approximation procedure. In the first step, the initial condition $u_{0}\in(BV\cap L^{\infty})(\R^{N})$ is approximated by smooth functions $u_{0}^{\d}\in C_{c}^{\infty}(\R^{N})$ such that $\|u_{0}^{\d}\|_2\le \|u_{0}\|_2$. In the second step, the driving signal $z$ is approximated by smooth driving signals $z^{(l)}$. In the third step, $A$ is approximated by  $A^\ve +\ve I$, where $I$ is the $N\times N$ identity matrix  and $A^{\ve}(\xi):=A\ast\vp^{\ve}(\xi)$ with  $\vp^{\ve}$ being a standard Dirac family. 

In conclusion, we consider the smooth solution $u^{(\ve,\d,l)}$ to
\begin{equation}
\begin{cases}
\partial_{t}u^{(\ve,\d,l)}+{\displaystyle \sum_{i=1}^{N}\partial_{x_{i}}F^{i}(u^{(\ve,\d,l)})(\dot{z}^{(l)})^{i}=\div(A^{\ve}(u^{(\ve,\d,l)})Du^{(\ve,\d,l)})+\ve\D u^{(\ve,\d,l)}
\quad\text{ in }\quad\R^{N}\times(0,T),}\\[2mm]
u^{(\ve,\d,l)}=u_{0}^{\d}\quad\text{ on }\quad\R^{N}\times\{0\}.
\end{cases}\label{eq:scl-1}
\end{equation}
The existence and uniqueness of $u^{(\ve,\d,l)}$ for each fixed $\ve, \d$ and $l$ is classical; see, for example, Volpert and Hudjaev \cite{VH69}.   The proof of the bounds in Theorem \ref{thm:existence} is based on establishing similar bounds for $u^{(\ve,\d,l)}$ and then passing to the limit.

\subsection*{Proof of Theorem \ref{thm:existence}.} Since the proof is long, we divide it into several steps.

\textbf{\textit{Step 1:}}\textit{ The approximating equation \eqref{eq:scl-1}.} The  kinetic function $\chi^{(\ve,\d,l)}(x,\xi,t):=\chi(u^{(\ve,\d,l)}(x,t),\xi)$ solves
\begin{equation}
\begin{cases}
\partial_{t}\chi^{(\ve,\d,l)}+f^{(l)}(\xi,t)\cdot D_{x}\chi^{(\ve,\d,l)}-\sum\limits _{i,j=1}^{d}a_{ij}^{\ve}(\xi)\partial_{x_{i}x_{j}}^{2}\chi^{(\ve,\d,l)}-\ve\D\chi^{(\ve,\d,l)}=\partial_{\xi}q^{(\ve,\d,l)}\ \text{in}\ \R^{N}\times\R\times(0,T),\\[1mm]
\chi^{(\ve,\d,l)}=\chi(u_{0}^{\d}(\cdot),\cdot)\ \text{ on }\ \R^{N}\times\R\times\{0\},
\end{cases}\label{eq:kinetic_form-approx}
\end{equation}
where 
\[
f^{(l)}(\xi,t):=F'(\xi)\dot{z}^{(l)}(t)
\]
and 
  $$q^{(\ve,\d,l)}=m^{(\ve,\d,l)}+n^{(\ve,\d,l)}$$
with the entropy dissipation and parabolic  measures  $m^{(\ve,\d,l)}$  and $n^{(\ve,\d,l)}$ given respectively by
\[
m^{(\ve,\d,l)}(x,\xi,t)=\d(\xi-u^{(\ve,\d,l)})\ve|Du^{(\ve,\d,l)}|^{2}
\]
and 
\[
n^{(\ve,\d,l)}(x,\xi,t)=\d(\xi-u^{(\ve,\d,l)})\sum_{k=1}^{N}\left(\sum_{i=1}^{N}\partial_{x_{i}}\b_{ik}^{\ve}(u^{(\ve,\d,l)})\right)^{2}.
\]
Moreover, it is shown in  \cite{VH69} that,  for all $t\ge0$ and $p\in[1,\infty]$, there exist  constants $C^{(l)}>0$ such that
\begin{align}
\|u^{(\ve,\d,l)}(t)\|_{L^{p}}  \le\|u_{0}\|_{p},\quad\|Du^{(\ve,\d,l)}(t)\|_{L^{1}}  \le BV(u_{0})\quad\text{and}\quad\label{eq:eps-bounds}
\|\partial_{t}u^{(\ve,\d,l)}(t)\|_{L^{1}}  \le C^{(l)}. 
\end{align}
Multiplying \eqref{eq:kinetic_form-approx} by $\xi$ and integrating yields 
\begin{align}\label{eq:measure_par_est-2}
\int_{\R^{N}\times\R\times\R_{+}}m^{(\ve,\d,l)}(x,\xi,t)dxd\xi dt+\int_{\R^{N}\times\R_{+}}\sum_{k=1}^{N}\left(\sum_{i=1}^{N}\partial_{x_{i}}\b_{ik}^{\ve}(u^{(\ve,\d,l)})\right)^{2}dxdt & \le\frac{1}{2}\|u_{0}^{\d}\|_{2}^{2}\le\frac{1}{2}\|u_{0}\|_{2}^{2}.
\end{align}
In view of Remark \ref{lem:kinetic-pathwise}, $u^{(\ve,\d,l)}$ is a pathwise entropy solution, that is, in the sense of distributions in $t$,
\begin{align}
\frac{d}{dt}(\vr^{(l)}\ast\chi^{(\ve,\d,l)})(y,\eta,t)= & \sum_{i,j=1}^{N}\int\chi^{(\ve,\d,l)}(x,\xi,t)a_{ij}^{\ve}(\xi)\partial_{x_{i}x_{j}}^{2}\vr^{(l)}(x,y,\xi,\eta,t)dxd\xi\nonumber \\
 & +\ve\int\chi^{(\ve,\d,l)}(x,\xi,t)\D\vr^{(l)}(x,y,\xi,\eta,t)dxd\xi \label{eq:kinetic-approx-2-1} \\
 & -\int\partial_{\xi}\vr^{(l)}(x,y,\xi,\eta,t)q^{(\ve,\d,l)}(x,\xi,t)dxd\xi.\nonumber 
\end{align}
Elementary calculations also give
\begin{align}
 & \sum_{i=1}^{N}\int\chi^{(\ve,\d,l)}(x,\xi,t)\sigma_{ik}^{\ve}(\xi)\partial_{x_{i}}\vr^{(l)}(x,y,\xi,\eta,t)dxd\xi\label{eq:chain-eps-delta-l}\\
 & =-\sum_{i=1}^{N}\int\partial_{x_{i}}\b_{ik}^{\ve}(u^{(\ve,\d,l)}(x,t))\vr^{(l)}(x,y,u^{(\ve,\d,l)}(x,t),\eta,t)dx.\nonumber 
\end{align}

\textbf{\textit{Step 2:}}\textit{ The singular degenerate limit $\ve\to0$.} ~In view of the estimates \eqref{eq:eps-bounds} and  \eqref{eq:measure_par_est-2}, there exists subsequences, which we  denote again by  $u^{(\ve,\d,l)}, m^{(\ve,\d,l)}$ and $n^{(\ve,\d,l)}$, such that, as $\ve \to 0$,
\[
u^{(\ve,\d,l)}\to u^{(\d,l)} \ \text{in }C([0,T];L^{1}(\R^{N})), \  m^{(\ve,\d,l)}\rightharpoonup m^{(\d,l)} \ \text{ and } \  n^{(\ve,\d,l)}\rightharpoonup n^{(\d,l)} \text{ weak $\star$},
\]
\[
\sum_{i=1}^{N}\partial_{x_{i}}\b_{ik}^{\ve}(u^{(\ve,\d,l)})\rightharpoonup\sum_{i=1}^{N}\partial_{x_{i}}\b_{ik}(u^{(\d,l)})\quad\text{in }L^{2}([0,T]\times\R^{N}), 
\]
and thus
\[
n^{(\ve,\d,l)}\rightharpoonup n^{(\d,l)}=\d(\xi-u^{(\d,l)})\sum_{k=1}^{N}\left(\sum_{i=1}^{N}\partial_{x_{i}}\b_{ik}(u^{(\d,l)})\right)^{2}.
\]
Passing to the limit in \eqref{eq:kinetic-approx-2-1} and \eqref{eq:chain-eps-delta-l} yields respectively \eqref{eq:kinetic} for $z=z^{(l)}$ and 
\[
\sum_{i=1}^{N}\int\chi^{(\d,l)}(x,\xi,t)\sigma_{ik}(\xi)\partial_{x_{i}}\vr^{(l)}(x,y,\xi,\eta,t)dxd\xi=\sum_{i=1}^{N}\int\partial_{x_{i}}\b_{ik}(u^{(\d,l)}(x,t))\vr^{(l)}(x,y,u^{(\d,l)}(x,t),\eta,t)dx.
\]
Using the lower semicontinuity with respect to weak convergence, we also get 
\[
\int_{0}^{T}\int_{\R^{N}\times\R}m^{(\d,l)}(x,\xi,t)dxd\xi dt+\int_{0}^{T}\int_{\R^{N}}\sum_{k=1}^{N}\left(\sum_{i=1}^{N}\partial_{x_{i}}\b_{ik}(u^{(\d,l)})\right)^{2}dxdt\le\frac{1}{2}\|u_{0}\|_{2}^{2}.
\]

\textbf{\textit{Step 3:}} \textit{The rough signal limit $l\to\infty$ and initial condition $\d\to0$.} Theorem \ref{thm:unique} yields that, as $l,m\to\infty$, 
\begin{align*}
\|u^{(\d,l)}-u^{(\d,m)}\|_{C([0,T];L^{1}(\R^{N}))}  \le C\|z^{(l)}-z^{(m)}\|^{1/2}_{C([0,T];\R^{N})} \to 0,
\end{align*}
that is, $u^{(\d,l)}$ is a Cauchy sequence in $C([0,T];L^{1}(\R^{N}))$ and thus has, as $l \to \infty$, a limit $u^{(\d)}$ in $C([0,T];L^{1}(\R^{N})).$ 

We then argue as in Step~$2$ to obtain a pathwise entropy solution $u^{(\d)}$ to \eqref{eq:scl} with initial condition $u_{0}^{\d}$. Again, Theorem \ref{thm:unique} yields that, as $\d_{1},\d_{2}\to0$,  
\begin{align*}
\|u^{(\d_{1})}-u^{(\d_{1})}\|_{C([0,T];L^{1}(\R^{N}))}  \le C\|u_{0}^{\d_{1}}-u_{0}^{\d_{2}}\|_{L^{1}(\R^{N})}
\end{align*}
and  arguing as before we obtain a pathwise entropy solution $u$. 

The bound \eqref{eq:measure_bound} follows easily by testing with $\xi^{[m+1]}=|\xi|^{m}\xi$ and a cut-off argument. Since the argument is routine, we leave the  details to the reader.

\section{Long-time behavior -- the proof of Theorem \ref{thm:ltb}\label{sec:ltb}}

The general approach is based on averaging techniques related to the classical averaging lemmata for scalar conservation laws. Typically, the proofs of averaging lemmata use Fourier transforms in space and time. In the stochastic context this is not possible due to the time-dependence of the flux. Therefore, we only use Fourier transforms in the spatial variable $x$. This technique was developed by Bouchut and Desvillettes \cite{BD99} and was used for semilinear stochastic scalar conservation laws by Debussche and Vovelle in \cite{DV13}. Although our proof follows the arguments of \cite{DV13} and more closely \cite{GS14-2}, new difficulties arise because of the second order term in \eqref{eq:scl-torus}. In particular, we have to adapt the important technical result (Lemma \ref{lem:integral_estimate-1}), which relies on the genuine nonlinearity condition \eqref{flux}. In what follows we are brief about parts similar to \cite{GS14-2} and we concentrate on the differences. Since the proof is rather long, we divide it in several subsections. We also remark that we use properties of the Brownian paths and our approach does not extend to general continuous time dependence.

\subsection*{Split-up of the solution}~Without loss of generality, we restrict to initial conditions with zero average, that is we assume that 
\begin{equation}
\int u_{0}(x)dx=0;\label{takis2}
\end{equation}
the case of non-zero spatial average can be easily reduced to this case. Moreover, in view of  \eqref{eq:gen_kinetic_unique-2} and the density of $L^{\infty}(\TT^{N})$ in $L^{1}(\TT^{N})$,  it is enough to consider $u_{0}\in L^{\infty}(\TT^{N})$. 

As a regularizing term we use the fractional Laplace operator (see \cite{BD99,DV13} for similar types of arguments) 
\begin{equation}
B:=(-\D)^{\alpha}+Id\ \text{ with}\ \a\in(0,1].\label{takis3}
\end{equation}

Let $S_{A_{\gamma}(\xi)}(s,t)$ denote the solution operator of 
\begin{equation}\label{eq:semigroup-eqn}
\partial_{t}v+\sum_{i}f^{i}(\xi)\dot{\b}^{i}(t)\partial_{x_{i}}v-\sum_{i,j=1}^{d}a_{ij}(\xi)\partial_{x_{i}x_{j}}^{2}v+\gamma Bv=0\ \text{in}\ \TT^{N}\times\R\times(s,\infty).
\end{equation}
Since $A(\xi):D^{2}=\sum_{i,j=1}^{d}a_{ij}(\xi)\partial_{x_{i}x_{j}}^{2}$ and $B$ commute, it is  immediate that, for all $f$ in the appropriate function space, 
\begin{align}
S_{A_{\gamma}(\xi)}(s,t)f(x) & =(e^{-(t-s)(A(\xi):D^{2}+\g B)}f)\left(x-f(\xi)(\b(t)-\b(s))\right),\label{eq:semigroup}\\
 & =(e^{-(t-s)A(\xi):D^{2}}e^{-(t-s)\g B}f)\left(x-f(\xi)(\b(t)-\b(s))\right),\nonumber 
\end{align}
where $e^{tA}$ denotes the solution semigroup to the operator $A$ and 
\[
f(\xi)(\b(t)-\b(s))=(f^{1}(\xi)(\beta^{1}(t)-\beta^{1}(s)),\dots,f^{N}(\xi)(\beta^{N}(t)-\beta^{N}(s))).
\]
For $n\in\Z^{N}$, the Fourier transform of $S_{A_{\gamma}(\xi)}$ corresponds to multiplication by 
\[
\exp\bigg({-if(\xi)(\beta(t)-\beta(s))\cdot n-\left(nA(\xi)n-\gamma(|n|^{2\alpha}+1)\right)(t-s)}\bigg).
\]
It follows from the variation of constants formula that, for $t\in[0,T]$ and $\vp\in C^{\infty}(\TT^{N})$,
\begin{align}
\int_{\TT^{N}}\vp(x)u(x,t)dx & =\int_{{\TT^{N}}\times\R}\vp(x)\chi(x,\xi,t)dxd\xi\nonumber \\
 & =\int_{{\TT^{N}}\times\R}\vp(x)S_{A_{\gamma}(\xi)}(0,t)\chi_{0}(x,\xi)dxd\xi+\int_{0}^{t}\int_{{\TT^{N}}\times\R}\gamma B(S_{A_{\gamma}(\xi)}^{*}(s,t)\vp)(x)\chi(x,\xi,s)dxd\xi ds\label{eq:splitup}\\
 & -\int_{0}^{t}\int_{{\TT^{N}}\times\R}\partial_{\xi}(S_{A_{\gamma}(\xi)}^{*}(s,t)\vp)(x)dq(x,\xi,s),\nonumber 
\end{align}
 where $S_{A_{\gamma}(\xi)}^{*}$ denotes the adjoint semigroup to $S_{A_{\gamma}(\xi)}$ and for simplicity we set again $q=m+n$. Noting that \eqref{eq:measure_bound} implies $q(\TT^N\times\{0\}\times \R_+)=0$ and that $\xi \mapsto (S_{A_{\gamma}(\xi)}^{*}(s,t)\vp)(x)$ is continuously differentiable on $\R\setminus\{0\}$, \eqref{eq:splitup} can be justified following the same arguments as in \cite[Appendix C]{GS14-2} in combination with dominated convergence based on \eqref{eq:measure_bound} and \eqref{takis data3}. We note that, in comparison to \cite{GS14-2}, the additional parabolic term of \eqref{eq:scl-torus} is represented in \eqref{eq:splitup} via the parabolic term in the semigroup $S_{A_{\gamma}(\xi)}$. 

Recall that, a.s. in $\omega$, $u\in C([0,\infty);L^{1}(\TT^{N}))$. Accordingly, in the sense of distributions in $x$, we have 
\begin{equation}
u(t)=u^{0}(t)+u^{1}(t)+Q(t),\label{takis5}
\end{equation}
where, for $\vp\in C^{\infty}(\TT^{N})$, 
\begin{align}
\<u^{0}(t),\vp\> & :=\int_{{\TT^{N}}\times\R}\vp(x)S_{A_{\gamma}(\xi)}(0,t)\chi_{0}(x,\xi)dxd\xi,\label{takis6}\\
\<u^{1}(t),\vp\> & :=\int_{0}^{t}\int_{{\TT^{N}}\times\R}\gamma B(S_{A_{\gamma}(\xi)}^{*}(s,t)\vp)(x)\chi(x,\xi,s)dxd\xi ds\label{takis7}\\
\<Q(t),\vp\> & :=-\int_{0}^{t}\int_{{\TT^{N}}\times\R}\partial_{\xi}(S_{A_{\gamma}(\xi)}^{*}(s,t)\vp)(x)q(x,\xi,s)dxd\xi ds.\label{takis8}
\end{align}
Next we estimate  \eqref{takis6}, \eqref{takis7} and \eqref{takis8} separately using averaging techniques. In the analysis we need a basic integral estimate which is proved in Appendix~\ref{lemma}. For its statement it is convenient to introduce, for each measurable $b:\R\to\R^{N}$, $a:\R\to\R_{+}$ and $f\in L^{2}$, the function $\phi(\cdot;a,b,f):\R^{N}\to\R$ given by 
\begin{equation}
\phi(w;a,b,f):=e^{-\frac{|w|^{2}}{\d}}\int_{\R}e^{ib(\xi)\cdot w-\d a(\xi)}f(\xi)d\xi.\label{takis9}
\end{equation}

\begin{lem} \label{lem:integral_estimate-1}Let $b:\R\to\R^{N},a:\R\to\R_{+}$ be measurable functions, such that, for all $\ve>0,z\in\R^{N}$ and some nondecreasing $\iota:[0,\infty)\to[0,\infty)$ with $\lim_{\ve\to0}i(\ve)=0$, 
\[
|\{\xi\in\R:|b(\xi)-z|^{2}+a(\xi)\le\ve\}|\le\iota(\ve).
\]
Then, for all $\d>0$ and $f\in L^{2}(\R)$, 
\[
\|\phi(\cdot;a,b,f)\|_{L^{2}}^{2}\le\frac{\sqrt{\d\pi}}{4}\int_{0}^{\infty}e^{-\frac{\tau}{4}}\iota(\frac{\tau}{\d})d\tau\|f\|_{2}^{2}.
\]
\end{lem}

\subsection*{The estimate of \texorpdfstring{$u^{0}$}{\^{u}0\}}}

Taking Fourier transforms in \eqref{takis6} yields, for each $n\in\Z^{N}$, 
\[
\hat{u}^{0}(n,t)=\int e^{-if(\xi)\b(t)\cdot n-(nA(\xi)n+\gamma(|n|^{2\alpha}+1))t}\hat{\chi}_{0}(n,\xi)d\xi.
\]
As in \cite{GS14-2} we have $\hat{u}^{0}(0,t)=0.$ For $n\in\Z^{N}\setminus\{0\}$, integrating in time, taking expectations and using the scaling properties of the Brownian paths, we find 
\begin{align*}
\E\int_{0}^{T}|\hat{u}^{0}(n,t)|^{2}dt & =\E\int_{0}^{T}\Big|\int e^{-if(\xi)\b({t})\cdot n-(nA(\xi)n+\gamma(|n|^{2\alpha}+1))t}\hat{\chi}_{0}(n,\xi)d\xi\Big|^{2}dt\\
 & =\int_{0}^{T}e^{-2\g(|n|^{2\alpha}+1)t}\E\Big|\int e^{-if(\xi)\b({t|n|^{2}})\cdot\frac{n}{|n|}-nA(\xi)nt}\hat{\chi}_{0}(n,\xi)d\xi\Big|^{2}dt\\
 & =\int_{0}^{T}\frac{e^{-2\g(|n|^{2\alpha}+1)t}}{\sqrt{2\pi|n|^{2}t}}\int\Big|e^{-\frac{|w|^{2}}{4|n|^{2}t}}\int e^{-if(\xi)\frac{n}{|n|}\cdot w-\frac{n}{|n|}A(\xi)\frac{n}{|n|}|n|^{2}t}\hat{\chi}_{0}(n,\xi)d\xi\Big|^{2}dwdt.
\end{align*}
\smallskip{}
Lemma \ref{lem:integral_estimate-1} with $\d=4|n|^{2}t$, $b(\xi)=f(\xi)\cdot\frac{n}{|n|}$, $a(\xi)=\frac{1}{4}\frac{n}{|n|}A(\xi)\frac{n}{|n|}$ and $i(\ve)\sim\ve^{\frac{\theta}{2}}$ gives
\begin{align*}
\E\int_{0}^{T}|\hat{u}^{0}(n,t)|^{2}dt & \le\frac{\sqrt{\d\pi}}{4}\int_{0}^{T}\frac{e^{-2\g(|n|^{2\alpha}+1)t}}{\sqrt{2\pi|n|^{2}t}}\int_{0}^{\infty}e^{-\frac{\tau}{4}}\iota(\frac{\tau}{\d})d\tau dt\|\hat{\chi}_{0}(n,\cdot)\|_{2}^{2}\\
 & \lesssim\d^{\frac{1-\theta}{2}}\int_{0}^{T}\frac{e^{-2\g(|n|^{2\alpha}+1)t}}{\sqrt{2\pi|n|^{2}t}}\int_{0}^{\infty}\tau^{\frac{\theta}{2}}e^{-\tau}d\tau dt\|\hat{\chi}_{0}(n,\cdot)\|_{2}^{2}\\
 & \lesssim\int_{0}^{T}e^{-2\g(|n|^{2\alpha}+1)t}(|n|^{2}t)^{-\frac{\theta}{2}}dt\|\hat{\chi}_{0}(n,\cdot)\|_{2}^{2}\\
 & \lesssim|n|^{-\theta}\int_{0}^{T}e^{-2\g(|n|^{2\alpha}+1)t}t{}^{-\frac{\theta}{2}}dt\|\hat{\chi}_{0}(n,\cdot)\|_{2}^{2}\\
 & \lesssim|n|^{-\theta}\g^{-\frac{2-\theta}{2}}(|n|^{2\alpha}+1)^{-\frac{2-\theta}{2}}\int_{0}^{\infty}e^{-t}t^{-\frac{\theta}{2}}dt\|\chi_{0}(\cdot,\cdot)\|_{2}^{2},
\end{align*}
and, hence, 
\begin{equation}
\E\int_{0}^{T}|\hat{u}^{0}(n,t)|^{2}dt\le C|n|^{-\theta}\g^{-\frac{2-\theta}{2}}(|n|^{2\alpha}+1)^{-\frac{2-\theta}{2}}\|\hat{\chi}_{0}(n,\cdot)\|_{2}^{2}\lesssim\gamma^{-\frac{2-\theta}{2}}\|\hat{\chi}_{0}(n,\cdot)\|_{2}^{2}.\label{eq:u_0-est}
\end{equation}
Combining the previous estimates, after summing over $n$, we obtain 
\begin{equation}
\E\int_{0}^{T}\|u^{0}(t)\|_{2}^{2}dt=\E\int_{0}^{T}\|\hat{u}^{0}(t)\|^{2}dt\lesssim\gamma^{-\frac{2-\theta}{2}}\|\chi_{0}(\cdot,\cdot)\|_{2}^{2}\lesssim\gamma^{-\frac{2-\theta}{2}}\|\chi_{0}(\cdot,\cdot)\|_{1}=\gamma^{-\frac{2-\theta}{2}}\|u_{0}\|_{1}.\label{takis1001}
\end{equation}

\subsection*{The estimate of \texorpdfstring{$u^{1}$}{\^{u}1\}} }

Let $\bar{\omega}_{n}:=\g(|n|^{2\a}+1)$. For each $n\in\Z^{N}$, the Fourier transform $\hat{u}^{1}(n,t)$ of $u^{1}(t)$ in $x$ is given by 
\begin{align*}
\hat{u}^{1}(n,t) & ={\bar{\omega}_{n}}\int_{0}^{t}\int e^{-if(\xi)(\b(t)-\b(s))\cdot n-(nA(\xi)n+\g(|n|^{2\alpha}+1))(t-s)}\hat{\chi}(n,\xi,s)d\xi ds\\
 & ={\bar{\omega}_{n}}\int_{0}^{t}e^{-\bar{\omega}_{n}(t-s)}\int e^{-if(\xi)(\b({t})-\b({s}))\cdot n-nA(\xi)n(t-s)}\hat{\chi}(n,\xi,s)d\xi ds.
\end{align*}
Integrating in $t$, taking expectation and using that $\int_{0}^{t}{\bar{\omega}_{n}}e^{-{\bar{\omega}_{n}}r}dr\le1$, we find 
\begin{align*}
 & \E\int_{0}^{T}|\hat{u}^{1}|^{2}(n,t)dt\\
= & \E\int_{0}^{T}\big|\int_{0}^{t}{\bar{\omega}_{n}}e^{-{\bar{\omega}_{n}}(t-s)}\int e^{-if(\xi)(\b(t)-\b(s))\cdot n-nA(\xi)n(t-s)}\hat{\chi}(n,\xi,s)d\xi ds\big|^{2}dt\\
= & \E\int_{0}^{T}\big|\int_{0}^{t}{\bar{\omega}_{n}}e^{-{\bar{\omega}_{n}}r}\int e^{-if(\xi)(\b({t})-\b({t-r}))\cdot n-nA(\xi)nr}\hat{\chi}(n,\xi,t-r)d\xi dr\big|^{2}dt\\
\le & \E\int_{0}^{T}\int_{0}^{t}\bar{\omega}_{n}e^{-{\bar{\omega}_{n}}r}\big|\int e^{-if(\xi)(\b({t})-\b({t-r}))\cdot n-nA(\xi)nr}\hat{\chi}(n,\xi,t-r)d\xi\big|^{2}drdt.
\end{align*}
Using that $\hat{\chi}$ is $\mcF_{t}$-adapted, that the increments $\b(t)-\b({t-r})$ are independent of $\mcF_{t-r}$ and the scaling properties of the Brownian motion, as in \cite{GS14-2}, we find 
\begin{align*}
 & \E\big|\int e^{-if(\xi)(\b({t})-\b({t-r}))\cdot n-nA(\xi)nr}\hat{\chi}(n,\xi,t-r)d\xi\big|^{2}\\
 & =\E\td\E\big|\int e^{-if(\xi)\b({|n|^{2}r})(\td\o)\cdot\frac{n}{|n|}-\left(\frac{n}{|n|}A(\xi)\frac{n}{|n|}\right)|n|^{2}r}\hat{\chi}(n,\xi,t-r)(\o)d\xi\big|^{2}\\
 & =\frac{1}{\sqrt{2\pi|n|^{2}r}}\E\int\big|e^{-\frac{|w|^{2}}{4|n|^{2}r}}\int e^{-if(\xi)w\cdot\frac{n}{|n|}-\left(\frac{n}{|n|}A(\xi)\frac{n}{|n|}\right)|n|^{2}r}\hat{\chi}(n,\xi,t-r)d\xi\big|^{2}dw,
\end{align*}
where $\td E$ is the expectation with respect to $\tilde\o$.
Using again Lemma \ref{lem:integral_estimate-1} with $\d=4|n|^{2}r$, $b(\xi)=f(\xi)\cdot\frac{n}{|n|}$ and $a(\xi)=\frac{1}{4}\left(\frac{n}{|n|}A(\xi)\frac{n}{|n|}\right)$ we get 
\begin{align*}
 & \int\big|e^{-\frac{|w|^{2}}{4|n|^{2}r}}\int e^{-if(\xi)w\cdot\frac{n}{|n|}-\left(\frac{n}{|n|}A(\xi)\frac{n}{|n|}\right)|n|^{2}r}\hat{\chi}(n,\xi,t-r)d\xi\big|^{2}dw\lesssim\sqrt{\d}\int_{0}^{\infty}e^{-\frac{\tau}{4}}\iota(\frac{\tau}{\d})d\tau\|\hat{\chi}(n,\cdot,t-r)\|_{2}^{2}\\[1mm]
 & \lesssim\d^{\frac{1-\theta}{2}}\int_{0}^{\infty}\tau^{\frac{\theta}{2}}e^{-\tau}d\tau\|\hat{\chi}(n,\cdot,t-r)\|_{2}^{2}\lesssim(|n|^{2}r)^{\frac{1-\theta}{2}}\|\hat{\chi}(n,\cdot,t-r)\|_{2}^{2},
\end{align*}
and, hence, 
\[
\E\big|\int e^{-if(\xi)(\b({t})-\b({t-r}))\cdot n-nA(\xi)nr}\hat{\chi}(n,t-r)d\xi\big|^{2}\lesssim(|n|^{2}r)^{-\frac{\theta}{2}}\E\|\hat{\chi}(n,\cdot,t-r)\|_{2}^{2}.
\]
Combining all the above estimates we find
\[
\E\int_{0}^{T}|\hat{u}^{1}|^{2}(n,t)dt\lesssim\int_{0}^{T}\int_{0}^{t}\bar{\omega}_{n}e^{-{\bar{\omega}_{n}}r}(|n|^{2}r)^{-\frac{\theta}{2}}\E\|\hat{\chi}(n,\cdot,t-r)\|_{2}^{2}drdt.
\]
Young's inequality then yields 
\[
\E\int_{0}^{T}|\hat{u}^{1}|^{2}(n,t)dt\lesssim\int_{0}^{T}\bar{\omega}_{n}e^{-{\bar{\omega}_{n}}r}(|n|^{2}r)^{-\frac{\theta}{2}}dr\int_{0}^{T}\E\|\hat{\chi}(n,\cdot,r)\|_{2}^{2}dr,
\]
and, in view of the fact that, for $\theta\in[0,1]$, 
\[
\int_{0}^{T}\bar{\o}_{n}e^{-\bar{\o}_{n}r}r{}^{-\frac{\theta}{2}}dr\le\bar{\o}_{n}^{1+\frac{\theta}{2}}\int_{\R}e^{-\bar{\o}_{n}r}(\bar{\o}_{n}r){}^{-\frac{\theta}{2}}dr=\bar{\o}_{n}^{\frac{\theta}{2}}\int_{\R}e^{-r}r{}^{-\frac{\theta}{2}}dr<\infty,
\]
we conclude that, for $n\in\Z^{N}\setminus\{0\}$, 
\begin{equation}
\E\int_{0}^{T}|\hat{u}^{1}|^{2}(n,t)dt\lesssim\g^{\frac{\theta}{2}}(|n|^{2\a-2}+|n|^{-2})^{\frac{\theta}{2}}\E\|\hat{\chi}(n,\cdot,\cdot)\|_{L^{2}(\R\times[0,T])}^{2}.\label{eq:u^1-est}
\end{equation}
Hence,
\[
\E\int_{0}^{T}\|u^{1}(t)\|_{2}^{2}dt\lesssim\g^{\frac{\theta}{2}}\E\int_{0}^{T}\|u(t)\|_{1}dt.
\]

\subsection*{The estimate of \texorpdfstring{$Q$}{Q}}

For $\l\ge0$ and $\vp\in L^{\infty}(\Omega\times[0,T];C^{\infty}(\TT^{N}))$ (in what follows, unless necessary, we do not display the dependence of $\vp$ in $\omega$ and $t$), let 
\begin{align*}
&<(-\D)^{\frac{\l}{2}}Q(t),\vp(t)> \\
&:=\int_{0}^{t}\int\partial_{\xi}(S_{A_{\gamma}(\xi)}(s,t)((-\D)^{\frac{\l}{2}}\vp(t)))(x)dq(x,\xi,s)\\
 & =\int_{0}^{t}\int_{\TT^{N}\times (\R\setminus\{0\})}\partial_{\xi}(S_{A_{\gamma}(\xi)}(s,t)((-\D)^{\frac{\l}{2}}\vp(t)))(x)dq(x,\xi,s)\\
 & =\int_{0}^{t}\int_{\TT^{N}\times (\R\setminus\{0\})}\left(f'(\xi)(\b_{t}-\b_{s})\cdot D+A'(\xi):D^{2}(t-s)\right)S_{A_{\gamma}(\xi)}^{*}(s,t)(-\D)^{\frac{\l}{2}}\vp(t)(x)dq(x,\xi,s),
\end{align*}
where the first equality follows from \eqref{eq:measure_bound} and the second equality is immediate from the definition of $S_{A_{\gamma}(\xi)}$ and $S_{A_{\gamma}(\xi)}^{*}$. 

In view of Lemma \ref{lem:reg_fractional_heat}, for any $\psi\in C^{\infty}(\TT^{N})$, we have 
\begin{align*}
 & \|D(S_{A_{\gamma}(\xi)}^{*}(s,t)(-\D)^{\frac{\l}{2}}\psi)\|_{\infty}\\
 & =\|D\left(e^{-\g(t-s)(\gamma B+A(\xi):D^{2})}(-\D)^{\frac{\l}{2}}\psi(\cdot)\right)\left(x-f(\xi)(\b({t})-\b({s}))\right)\|_{\infty}\\
 & =e^{-\g(t-s)}\|\left(e^{-\g(t-s)A(\xi):D^{2}}De^{-\g(t-s)(-\D)^{\a}}(-\D)^{\frac{\l}{2}}\psi(\cdot)\right)\left(x-f(\xi)(\b({t})-\b({s}))\right)\|_{\infty}\\
 & \le e^{-\g(t-s)}\|De^{-\g(t-s)(-\D)^{\a}}(-\D)^{\frac{\l}{2}}\psi\|_{\infty}
 \lesssim e^{-\g(t-s)}(\gamma(t-s))^{-\frac{\l+1}{2\a}}\|\psi\|_{\infty}.
\end{align*}
Again using Lemma \ref{lem:reg_fractional_heat} we have
\begin{align*}
 & \|\partial_{x_{i}x_{j}}^{2}(S_{A_{\gamma}(\xi)}^{*}(s,t)(-\D)^{\frac{\l}{2}}\psi)\|_{\infty}\\
 & =\|\partial_{x_{i}x_{j}}^{2}\left(e^{-(t-s)(\g B+A(\xi):D^{2})}(-\D)^{\frac{\l}{2}}\psi(\cdot)\right)\left(x-f(\xi)(\b({t})-\b({s}))\right)\|_{\infty}\\
 & =e^{-\g(t-s)}\|\left(e^{-(t-s)A(\xi):D^{2}}\partial_{x_{i}x_{j}}^{2}e^{-(t-s)\g(-\D)^{\a}}(-\D)^{\frac{\l}{2}}\psi(\cdot)\right)\left(x-f(\xi)(\b({t})-\b({s}))\right)\|_{\infty}\\
 & \le e^{-\g(t-s)}\|\partial_{x_{i}x_{j}}^{2}e^{-(t-s)\g(-\D)^{\a}}(-\D)^{\frac{\l}{2}}\psi(\cdot)\|_{\infty}
 \lesssim e^{-\g(t-s)}(\gamma(t-s))^{-\frac{\l+2}{2\a}}\|\psi\|_{\infty}.
\end{align*}
In view of \eqref{takis data3}, it follows that, for all $\vp\in L^{\infty}(\Omega\times[0,T];C^{\infty}(\TT^{N})),$ 
\begin{align*}
 & \E\int_{0}^{T}\<(-\D)^{\frac{\l}{2}}Q(t),\vp(t)\>dt\\
 & \le C\|\vp\|_{\infty}\E\int_{0}^{T}\int_{0}^{t}e^{-\g(t-s)}\left(\gamma^{-\frac{\l+1}{2\a}}|\b({t})-\b({s})|(t-s)^{-\frac{\l+1}{2\a}}+\gamma^{-\frac{\l+2}{2\a}}(t-s){}^{1-\frac{\l+2}{2\a}}\right)\\
 & \int_{\TT^{N}\times (\R\setminus\{0\})}(1+|\xi|^{p_{1}}+|\xi|^{p_{2}})dq(x,\xi,s)dt\\
 & \le C\|\vp\|_{\infty}\E\int_{0}^{T}\int_{0}^{t}e^{-\g(t-s)}\left(\gamma^{-\frac{\l+1}{2\a}}|\b({t})-\b({s})|(t-s)^{-\frac{\l+1}{2\a}}+\gamma^{-\frac{\l+2}{2\a}}(t-s){}^{1-\frac{\l+2}{2\a}}\right)\\
 & \int_{{\TT^{N}}\times\R}(1+|\xi|^{p_{1}}+|\xi|^{p_{2}})dq(x,\xi,s)dt
\end{align*}
At this point we need to argue as in the previous step by taking conditional expectation and using that $\b({t})-\b({t-r})$ is independent of $\mathcal{F}_{t-r}$. Since $q$ is only a measure, to make the following argument rigorous it is necessary to perform another approximation. This is done in \cite{GS14-2} and we omit the details here.

Using the independence of $\b({t})-\b({s})$ from $\mcF_{s}$ and $\mcF_{s}$-measurability of $q$ we find 
\[
\E|\b({t})-\b({s})|q(s)=\E[\E|\b({t})-\b({s})|q(s)|\mcF_{s}]=\E|\b({t})-\b({s})|\E q(s)=\sqrt{t-s}\E q(s),
\]
and employing once more  Young's inequality we obtain 
\begin{align*}
 & \E\int_{0}^{T}\int_{0}^{t}e^{-\g(t-s)}\left(|\b({t})-\b({s})|\gamma^{-\frac{\l+1}{2\a}}(t-s)^{-\frac{\l+1}{2\a}}+\gamma^{-\frac{\l+2}{2\a}}(t-s){}^{1-\frac{\l+2}{2\a}}\right)\int_{{\TT^{N}}\times\R}(1+|\xi|^{p_{1}}+|\xi|^{p_{2}})dq(x,\xi,s)dt\\
 & =\int_{0}^{T}\int_{0}^{t}e^{-\g(t-s)}\left(\gamma^{-\frac{\l+1}{2\a}}(t-s)^{\frac{1}{2}-\frac{\l+1}{2\a}}+\gamma^{-\frac{\l+2}{2\a}}(t-s){}^{1-\frac{\l+2}{2\a}}\right)\E\int_{{\TT^{N}}\times\R}(1+|\xi|^{p_{1}}+|\xi|^{p_{2}})dq(x,\xi,s)dt\\
 & \le\int_{0}^{T}e^{-\g t}\left(\gamma^{-\frac{\l+1}{2\a}}t{}^{\frac{1}{2}-\frac{\l+1}{2\a}}+\gamma^{-\frac{\l+2}{2\a}}t{}^{1-\frac{\l+2}{2\a}}\right)dt\int_{0}^{T}\E\int_{{\TT^{N}}\times\R}(1+|\xi|^{p_{1}}+|\xi|^{p_{2}})dq(x,\xi,s)dt.
\end{align*}
In conclusion, 
\begin{align*}
\E\int_{0}^{T}\<(-\D)^{\frac{\l}{2}}Q(t),\vp(t)\>dt\lesssim & \|\vp\|_{\infty}\int_{0}^{T}e^{-\g t}\left(\gamma^{-\frac{\l+1}{2\a}}t{}^{\frac{1}{2}-\frac{\l+1}{2\a}}+\gamma^{-\frac{\l+2}{2\a}}t{}^{1-\frac{\l+2}{2\a}}\right)dt\\
 & \E\int_{0}^{T}\int_{{\TT^{N}}\times\R}(1+|\xi|^{p_{1}}+|\xi|^{p_{2}})dq(x,\xi,s).
\end{align*}
Moreover note that, if $\d>-1$, then 
\[
\int_{0}^{T}t{}^{\d}e^{-\g t}dt=\g^{-\d}\int_{0}^{T}(\g t){}^{\d}e^{-\g t}dt=\g^{-\d-1}\int_{\R_{+}}t{}^{\d}e^{-t}dt\le C\g^{-\d-1},
\]
and for $\d=\frac{1}{2}-\mu_{\a,\l}^{(1)}=\frac{1}{2}-\frac{\l+1}{2\a}$, $\d=1-\mu_{\a,\l}^{(2)}=1-\frac{\l+2}{2\a}$ and assuming $\mu_{\a,\l}^{(2)}<2$ (note that this also implies $\mu_{\a,\l}^{(1)}<\frac{3}{2}$,)  we get
\[
\int_{0}^{T}e^{-\g t}\left(\gamma^{-\mu_{\a,\l}^{(1)}}t{}^{\frac{1}{2}-\mu_{\a,\l}^{(1)}}+\gamma^{-\mu_{\a,\l}^{(2)}}t{}^{1-\mu_{\a,\l}^{(2)}}\right)dt\le C(\g^{-\frac{3}{2}}+\g^{-2}).
\]
We use next \eqref{eq:measure_bound} and get, in view of all the above, 
\begin{align}
 & \E\int_{0}^{T}\<(-\D)^{\frac{\l}{2}}Q(t),\vp(t)\>dt\lesssim\|\vp\|_{\infty}(\g^{-\frac{3}{2}}+\g^{-2})(\|u_{0}\|_{2}^{2}+\|u_{0}\|_{p_{1}+2}^{p_{1}+2}+\|u_{0}\|_{p_{2}+2}^{p_{2}+2}).\label{eq:Q_est}
\end{align}
For now it is enough to take $\l=0$. Then $\mu_{\a,0}^{(1)}<\frac{3}{2}$, $\mu_{\a,0}^{(2)}<2$ is satisfied if $\a>\frac{1}{2}$, and, for $\g\le1$, we obtain 
\begin{align*}
 & \E\int_{0}^{T}\<Q(t),\vp(t)\>dt\lesssim\g^{-2}\|\vp\|_{\infty}(\|u_{0}\|_{2}^{2}+\|u_{0}\|_{p_{1}+2}^{p_{1}+2}+\|u_{0}\|_{p_{2}+2}^{p_{2}+2}).
\end{align*}

\subsection*{The proof of Theorem \ref{thm:ltb}}~We conclude the proof as in \cite{GS14-2}. Note that as compared to \cite{GS14-2} the constants change due to the changed constants in the estimate of $Q.$ We obtain 
\begin{align*}
\E\|u(T)\|_{1} & \lesssim T^{-\frac{1}{2}+a(\frac{2-\theta}{4})}\|u_{0}\|_{1}^{\frac{1}{2}}+2T^{-\frac{a\theta}{2}}+T^{2a-1}(\|u_{0}\|_{2}^{2}+\|u_{0}\|_{p_{1}+2}^{p_{1}+2}+\|u_{0}\|_{p_{2}+2}^{p_{2}+2}).
\end{align*}
and letting $a=\frac{2}{4+\t}$, for $T\ge1$, we get
\begin{align*}
\E\|u(T)\|_{1} & \lesssim T^{-\frac{\theta}{4+\theta}}\left(\|u_{0}\|_{1}^{\frac{1}{2}}+1+\|u_{0}\|_{2}^{2}+\|u_{0}\|_{p_{1}\vee p_{2}+2}^{p_{1}\vee p_{2}+2}\right);
\end{align*}
note that the rate is independent of the choice of $\a$. The proof is concluded as in \cite{GS14-2}.

\section{Regularity (proof of Theorem \ref{thm:reg})\label{sec:reg}}

The remaining argument is precisely the same as in \cite{GS14-2}. First, assuming $\chi = \chi(u) \in L^2(\R\times[0,T]\times\Omega;H^\tau(\R^N))$ for some $\tau\in[0,1]$ this implies
\begin{equation}
\E\int_{0}^{T}\|u(t)\|_{W^{\l,1}}dt\le C(1+\|u_{0}\|_{p_{1}+2}^{p_{1}+2}+\|u_{0}\|_{p_{2}+2}^{p_{2}+2}),\label{eq:reg_prop-1}
\end{equation}
for $\l$ satisfying the constraint 
\[
\l<(4\a-2)\wedge(\t(1-\a)+\tau).
\]
Note that the constants change as compared to \cite{GS14-2} due to the changed constants in the estimation of $Q$. Maximizing the right hand side yields $\a=\frac{\t+2+\tau}{\t+4}\in(0,1)$ and we obtain 
\[
\l<4\a-2=\frac{2\t}{\t+4}+\frac{4\tau}{\t+4}.
\]
As in \cite{GS14-2} this bound is now bootstrapped, which yields that \eqref{eq:reg_prop-1} holds for all $\l\in\left(0,\frac{2\t}{\t+2}\right)$. The proof is then concluded as in \cite{GS14-2}. 

\appendix

\section{Convolution error estimates}\label{sec:err-est}

We study here the behavior of the several error terms introduced in the proof of the comparison principle for pathwise entropy solutions. 

The first type of errors, which are studied in the first subsection below, are  due to the convolution along characteristics,  that is, when we replace $\chi(x,\xi,t)$ by $\td\chi(x,\xi,t):=\chi(x+f(\xi)z_{t},\xi,t)$ and $\td\chi$ by its approximation $\chi\ast\vr_{\ve}^{s}$. We then consider the parabolic error terms arising from doubling of variables/mollification of the velocity variable $\xi$ in the second subsection. In the last subsection we estimate the hyperbolic error which a is consequence of the failure of the right hand side in \eqref{eq:informal_transform} to be the $\xi$-derivative of a nonnegative measure.
\subsection*{Convolution along characteristics\label{sub:Convolution-along-characteristic}}
Let 
\[
\vr_{\ve}^{s}(x,y,\xi,t):=\vr_{\ve}^{s}(x-y+f(\xi)z(t)),
\]
where $\vr_{\ve}^{s}$ is a standard Dirac family. 
\begin{lem}
\label{lem:convolution_est}Let $u\in BV(\R^{N})$. Then, for all $t\in[0,T]$, $p\in[1,\infty)$,
\[
\|\left(\chi\ast\vr_{\ve}^{s}\right)(y,\xi,t)-\chi(y-f(\xi)z(t),\xi)\|_{L^{p}(\R^{N}\times\R)}\le\ve BV(u).
\]
\end{lem}
\begin{proof}
Using Hölder's inequality and $u\in BV$ we find
\begin{align*}
 & \|\left(\chi\ast\vr_{\ve}^{s}\right)(y,\xi,t)-\chi(y-f(\xi)z(t),\xi)\|_{L^{p}(\R^{N}\times\R)}^{p}\\
 & =\int\Big|\int\chi(x,\xi)\vr_{\ve}^{s}(x,y,\xi,t)dx-\chi(y-f(\xi)z(t),\xi)\Big|^{p}dyd\xi\\
 & =\int\Big|\int[\chi(x,\xi)-\chi(y-f(\xi)z(t),\xi)]\vr_{\ve}^{s}(x,y,\xi,t)dx\Big|^{p}dyd\xi\\
 & \le\int\Big|\left(\int\vr_{\ve}^s(x,y,\xi,t)dx\right)^{\frac{1}{q}}\left(\int|\chi(x,\xi)-\chi(y-f(\xi)z(t),\xi)|^{p}\vr_{\ve}^{s}(x,y,\xi,t)dx\right)^{\frac{1}{p}}\Big|^{p}dyd\xi\\
 & \le\int\int|\chi(x,\xi)-\chi(y-f(\xi)z(t),\xi)|^{p}\vr_{\ve}^{s,0}(x-y+f(\xi)z(t))dxdyd\xi\\
 & =\int\int|\chi(x,\xi)-\chi(y,\xi)|\vr_{\ve}^{s}(x-y)dxdyd\xi\\
 & =\int|u(x)-u(y)|\vr_{\ve}^{s}(x-y)dxdy\\
 & \le\ve BV(u).
\end{align*}
\end{proof}
\begin{lem}
\label{lem:convolution_est-1} Let $u\in(BV\cap L^{\infty})(\R^{N})$. Then, for all $t\in[0,T]$,
\[
\|\chi(y-f(\xi)z(t),\xi)-\chi(y,\xi)\|_{L^{1}(\R^{N}\times\R)}\le\|f(u)\|_{\infty}|z(t)|BV(u).
\]
\end{lem}
\begin{proof}
Due to \cite[Lemma C.1]{GPS15} we have that
\[
\|u\|_{BV}=\int_{\R}BV(\chi(\cdot,\xi))d\xi,
\]
with $\chi(x,\xi):=\chi(u(x),\xi)$. It follows that
\begin{align*}
\int|\chi(y-f(\xi)z(t),\xi)-\chi(y,\xi)|dyd\xi  \le\int BV(\chi(\cdot,\xi))|f(\xi)z(t)|d\xi \le|z(t)|\|f(u)\|_{\infty}\int BV(\chi(\cdot,\xi))d\xi. 
\end{align*}
\end{proof}
Let 
\begin{align*}
\tilde{G}(t):= & -2\int\chi^{(1)}(y-f(\xi)z^{(1)}(t),\xi,t)\chi^{(2)}(y-f(\xi)z^{(1)}(t),\xi,t)dyd\xi\\
 & +\int\sgn(\eta)\chi^{(1)}(y-f(\xi)z^{(1)}(t),\xi,t)dyd\xi+\int\sgn(\eta)\chi^{(2)}(y-f(\xi)z^{(2)}(t),\xi,t)dyd\xi\\
= & -2\int\chi^{(1)}(y-f(\xi)(z^{(1)}(t)-z^{(2)}(t)),\xi,t)\chi^{(2)}(y,\xi,t)dyd\xi\\
 & +\int\sgn(\eta)\chi^{(1)}(y,\xi,t)dyd\xi+\int\sgn(\eta)\chi^{(2)}(y,\xi,t)dyd\xi
\end{align*}
Lemma \ref{lem:convolution_est}, with $G_{\ve}$ as in Section \ref{sec:unique}, yields
\begin{align*}
|G_{\ve}(t)-\td G(t)| & \le\ve(BV(u^{(1)}(t))+BV(u^{(2)}(t))).
\end{align*}
and,  with $G$ as in Section~\ref{sec:unique},  Lemma \ref{lem:convolution_est-1} implies
\begin{align*}
|\td G(t)-G(t)| & \le\|f(u^{(1)})\|_{\infty}|z^{(1)}(t)-z^{(2)}(t)|BV(u^{(1)}(t)).
\end{align*}
In conclusion, we have:
\begin{lem}
\label{lem:convolution_est-1-1} If $u\in(BV\cap L^{\infty})(\R^{N})$, then for all $t\in [0,T]$,
\begin{align*}
|G_{\ve}(t)-G(t)| & \le\left(\ve+\|f(u^{(1)}(t))\|_{\infty}|z^{(1)}(t)-z^{(2)}(t)|\right)(BV(u^{(1)}(t))+BV(u^{(2)}(t))).
\end{align*}

\end{lem}

\subsection*{Parabolic Error\label{sec:par_err}}
We study the parabolic error $Err_{\ve,\psi,\d}^{par}$ occurring in the proof of Theorem \ref{thm:unique}.
\begin{lem}
\label{lem:par_err}Let $s,t\in[0,T]$, $s<t$, $\psi\in C_{c}^{\infty}(\R^{N+1})$, $\ve>0$. Then, as $\d\to 0$, $ \int_{s}^{t}Err_{\ve,\psi,\d}^{par}(r)dr\to0.$
\end{lem}
\begin{proof}
Straightforward calculations lead to 
\begin{align}
Err_{\ve,\psi,\d}^{par}(r)= & 2\sum_{i,j=1}^{N}\int\int\psi(y,\eta)\chi^{(1)}(x,\xi,r)\chi^{(2)}(x',\xi',r)\left(a_{ij}(\xi)-2\sum_{k=1}^{N}\sigma_{ik}(\xi)\s_{kj}(\xi')+a_{ij}(\xi')\right)\nonumber \\
 & \partial_{x_{i}}\vr_{\ve,\d}^{(1)}(x,y,\xi,\eta,r)\partial_{x_{j}'}\vr_{\ve,\d}^{(2)}(x',y,\xi',\eta,r)dxd\xi dx'd\xi'dyd\eta\nonumber \\
= & 2\sum_{i,j=1}^{N}\int\psi(y,\eta)\left(\int\chi^{(1)}(x,\xi,r)a_{ij}(\xi)\partial_{x_{i}}\vr_{\ve,\d}^{(1)}(x,y,\xi,\eta,r)dxd\xi\right)\nonumber \\
 & \left(\int\chi^{(2)}(x',\xi',r)\partial_{x_{j}'}\vr_{\ve,\d}^{(2)}(x',y,\xi',\eta,r)dx'd\xi'\right)dyd\eta\nonumber\\
 & -4\sum_{i,j=1}^{N}\sum_{k=1}^{N}\int\psi(y,\eta)\left(\int\chi^{(1)}(x,\xi,r)\sigma_{ik}(\xi)\partial_{x_{i}}\vr_{\ve,\d}^{(1)}(x,y,\xi,\eta,r)dxd\xi\right)\nonumber \\
 & \left(\int\chi^{(2)}(x',\xi',r)\s_{kj}(\xi')\partial_{x_{j}'}\vr_{\ve,\d}^{(2)}(x',y,\xi',\eta,r)dx'd\xi'\right)dyd\eta\nonumber \\
 & +2\sum_{i,j=1}^{N}\int\psi(y,\eta)\left(\int\chi^{(1)}(x,\xi,r)\partial_{x_{i}}\vr_{\ve,\d}^{(1)}(x,y,\xi,\eta,r)dxd\xi\right)\nonumber \\
 & \left(\int\chi^{(2)}(x',\xi',r)a_{ij}(\xi')\partial_{x_{j}'}\vr_{\ve,\d}^{(2)}(x',y,\xi',\eta,r)dx'd\xi'\right)dyd\eta.\nonumber 
\end{align}
We now aim to take the limit $\d\to0$ in each of these three terms. Since the first and third terms are similar, here we only give the details for the first two.

Note that, since $\psi$ has compact support, the integration in $(x,y)$ and $(x',y')$ is taking place over bounded sets. This together with $a_{ij}\in L_{loc}^{\infty}$ yields that, in the limit $\d \to 0$,  and in  $L_{loc}^{2}(\R)$
\begin{align*}
\int\chi^{(1)}(x,\xi,r)a_{ij}(\xi)\partial_{x_{i}}\vr_{\ve,\d}^{(1)}(x,y,\xi,\eta,r)dxd\xi
 &=\left(\int\chi^{(1)}(x,\cdot,r)a_{ij}(\cdot)\partial_{x_{i}}\vr_{\ve}^{s,(1)}(x,y,\cdot,r)dx\right)\ast\vr_{\d}^{v}(\eta) \\
 &\to \int\chi^{(1)}(x,\eta,r)a_{ij}(\eta)\partial_{x_{i}}\vr_{\ve}^{s,(1)}(x,y,\eta,r)dx.
\end{align*}
Since all functions are locally bounded, dominated convergence  implies that, as $\d\to0$,
\begin{align*}
 & \sum_{i,j=1}^{N}\int_{s}^{t}\int\psi(y,\eta)\left(\int\chi^{(1)}(x,\xi,r)a_{ij}(\xi)\partial_{x_{i}}\vr_{\ve,\d}^{(1)}(x,y,\xi,\eta,r)dxd\xi\right)
  \left(\int\chi^{(2)}(x',\xi',r)\partial_{x_{j}'}\vr_{\ve,\d}^{(2)}(x',y,\xi',\eta,r)dx'd\xi'\right)dyd\eta dr\\
 & \to \sum_{i,j=1}^{N}\int_{s}^{t}\int\psi(y,\eta)\chi^{(1)}(x,\eta,r)\chi^{(2)}(x',\eta,r)a_{ij}(\eta)
  \partial_{x_{i}}\vr_{\ve}^{s,(1)}(x,y,\eta,r)\partial_{x_{j}'}\vr_{\ve}^{s,(2)}(x',y,\eta,r)dx'dxdyd\eta dr,
\end{align*}
Similarly,  as $\d \to0$,
\begin{align*}
 & \sum_{i,j,k=1}^{N}
 \int_{s}^{t}\int\psi(y,\eta)\chi^{(1)}(x,\xi,r)\sigma_{ik}(\xi)\partial_{x_{i}}\vr_{\ve,\d}^{(1)}(x,y,\xi,\eta,r)dxd\xi
 \int\chi^{(2)}(x',\xi',r)\s_{kj}(\xi')\partial_{x_{j}'}\vr_{\ve,\d}^{(2)}(x',y,\xi',\eta,r)dx'd\xi'dyd\eta dr\\
 & \to \sum_{i,j,k=1}^{N}
 \int_{s}^{t}\int\psi(y,\eta)\chi^{(1)}(x,\eta,r)\sigma_{ik}(\eta)\partial_{x_{i}}\vr_{\ve,\d}^{s,(1)}(x,y,\eta,r)dx
  \int\chi^{(2)}(x',\eta,r)\s_{kj}(\eta)\partial_{x_{j}'}\vr_{\ve,\d}^{s,(2)}(x',y,\eta,r)dx'dyd\eta dr\\
 & =\sum_{i,j=1}^{N}\int_{s}^{t}\int\psi(y,\eta)\chi^{(1)}(x,\eta,r)\chi^{(2)}(x',\eta,r)a_{ij}(\eta)\partial_{x_{i}}\vr_{\ve,\d}^{s,(1)}(x,y,\eta,r)\partial_{x_{j}'}\vr_{\ve}^{s,(2)}(x',y,\eta,r)dxdx'dyd\eta dr.
\end{align*}
\end{proof}

\subsection*{The localization Error\label{sec:loc_err}} We study here the error terms appearing in Section~\ref{sec:unique} due to the localization in the $(y,\eta)$ variables, that is $Err_{\ve,\psi,\d}^{loc,(i)}$, $i=1,2,3,4$. Since their analysis is similar, we concentrate on $Err_{\ve,\psi,\d}^{loc,(1)}$ and $Err_{\ve,\psi,\d}^{loc,(3)}$.

Recall 
\begin{align*}
Err_{\ve,\psi,\d}^{loc,(1)}(r) & =\int(\sgn\ast\vr_{\d}^{v})(\eta)\left(\partial_{y_{i}y_{j}}\psi(y,\eta)\sum_{i,j=1}^{N}\int\chi^{(2)}(x',\xi',r)a_{ij}(\xi')\vr_{\ve,\d}^{(2)}(x',y,\xi',\eta,r) dx'd\xi'\right)dyd\eta\\
Err_{\ve,\psi,\d}^{loc,(3)}(r) & =-\sum_{i,j=1}^{N}\int\partial_{y_{i}}\psi(y,\eta)\chi^{(1)}(x,\xi,r)\chi^{(2)}(x',\xi',r)a_{ij}(\xi')\vr^{(1)}(x,y,\xi,\eta,r)\partial_{x'_{j}}\vr^{(2)}(x',y,\xi',\eta,r)dxd\xi dx'd\xi'dyd\eta
\end{align*}
and set 
\begin{align*}
Err_{\ve,\psi}^{loc,(1)}(r) & =\sum_{i,j=1}^{N}\int\sgn(\eta)\partial_{y_{i}y_{j}}\psi(y,\eta)\chi^{(2)}(x',\eta,r)a_{ij}(\eta)\vr_{\ve}^{s,(2)}(x',y,\eta,r)dx'dyd\eta\\
Err_{\ve,\psi}^{loc,(3)}(r) & =-\sum_{i,j=1}^{N}\int\partial_{y_{i}}\psi(y,\eta)\chi^{(1)}(x,\eta,r)\chi^{(2)}(x',\eta,r)a_{ij}(\eta)\vr^{s,(1)}(x,y,\eta,r)\partial_{x'_{j}}\vr^{s,(2)}(x',y,\eta,r)dxdx'dyd\eta.
\end{align*}

\begin{lem}
\label{lem:err-loc-2}For all $s<t$, $s,t\in[0,T]$, as $\d\to0$,  we have 
\begin{align*}
\int_{s}^{t}Err_{\ve,\psi,\d}^{loc,(1)}(r)dr  \to\int_{s}^{t}Err_{\ve,\psi}^{loc,(1)}(r)dr \ \  \text{and} \ \
\int_{s}^{t}Err_{\ve,\psi,\d}^{loc,(3)}(r)dr  \to\int_{s}^{t}Err_{\ve,\psi}^{loc,(3)}(r)dr.
\end{align*}
\end{lem}
\begin{proof}
Noting that $\psi$ has compact support and $a_{ij}\in L_{loc}^{\infty}$, the proof is a simple consequence of convergence of mollifications along Dirac families.
\end{proof}

\begin{lem}
\label{lem:err-loc}Let $\psi_{R}$ as in \eqref{eq:psi_loc}. Then, as $R\to \infty$,  
\begin{align*}
 & \int_{s}^{t}Err_{\ve,\psi_{R}}^{loc,(1)}(r)dr\to0\quad\text{and }\int_{s}^{t}Err_{\ve,\psi_{R}}^{loc,(3)}(r)dr\to0.
\end{align*}
\end{lem}
\begin{proof}
The convergence of $Err_{\ve,\psi_{R}}^{loc,(1)}$ is a consequence of the dominated convergence theorem. We first observe that
\begin{align*}
 & |\sgn(\eta)\partial_{y_{i}y_{j}}\psi_{R}(y,\eta)\chi^{(2)}(x',\eta,r)a_{ij}(\eta)\vr_{\ve}^{s,(2)}(x',y,\eta,r)|\\[1.5mm]
&  \le C|\chi^{(2)}|(x',\eta,r)a_{ij}(\eta)\vr_{\ve}^{s,(2)}(x',y,\eta,r)
  \le C\|a_{ij}\|_{L_{loc}^{\infty}}|\chi^{(2)}|(x',\eta,r)\vr_{\ve}^{s,(2)}(x',y,\eta,r)
\end{align*}
and
\begin{align*}
  \int_{s}^{t}\int|\chi^{(2)}|(x',\eta,r)\vr_{\ve}^{s,(2)}(x',y,\eta,r)dydx'd\eta dr
 =\int_{s}^{t}\int|\chi^{(2)}|(x',\eta,r)dx'd\eta dr
 =\int_{s}^{t}\int|u^{(2)}|(x',r)dx'dt<\infty.
\end{align*}
Moreover, as $R\to\infty$ and for all $(y,\eta)\in\R^{N+1}$ and $r\in[0,T]$,
\[
\sgn(\eta)\partial_{y_{i}y_{j}}\psi_{R}(y,\eta)\chi^{(2)}(x',\eta,r)a_{ij}(\eta)\vr_{\ve}^{s,(2)}(x',y,\eta,r)\to0, 
\]
and the claim now follows again from dominated convergence.

For the convergence of $Err_{\ve,\psi_{R}}^{loc,(3)}$ we observe that  
\begin{align*}
 & \partial_{y_{i}}\psi(y,\eta)\chi^{(1)}(x,\eta,r)\chi^{(2)}(x',\eta,r)a_{ij}(\eta)\vr^{s,(1)}(x,y,\eta,r)\partial_{x'_{j}}\vr^{s,(2)}(x',y,\eta,r)\\
 & \le C|\chi^{(1)}|(x,\eta,r)|\chi^{(2)}|(x',\eta,r)|a_{ij}|(\eta)|\vr^{s,(1)}|(x,y,\eta,r)|\partial_{x'_{j}}\vr^{s,(2)}|(x',y,\eta,r)\\
 & \le C\|a_{ij}\|_{L_{loc}^{\infty}}|\chi^{(1)}|(x,\eta,r)|\chi^{(2)}|(x',\eta,r)|\vr^{s,(1)}|(x,y,\eta,r)|\partial_{x'_{j}}\vr^{s,(2)}|(x',y,\eta,r)
\end{align*}
and 
\begin{align*}
 & \int_{s}^{t}\int|\chi^{(1)}|(x,\eta,r)|\chi^{(2)}|(x',\eta,r)|\vr^{s,(1)}|(x,y,\eta,r)|\partial_{x'_{j}}\vr^{s,(2)}|(x',y,\eta,r)dxdx'dyd\eta dr\\
 & \le\int_{s}^{t}\int|\chi^{(2)}|(x',\eta,r)|\partial_{x'_{j}}\vr^{s,(2)}|(x',y,\eta,r)dx'dyd\eta dr\\
 & \le C\int_{s}^{t}\int|\chi^{(2)}|(x',\eta,r)dx'd\eta dr
  \le C\int_{s}^{t}\int|u^{(2)}|(x',r)dx'dr
  <\infty.
\end{align*}
Since,  as  $R\to\infty$ and for all $(y,\eta)\in\R^{N+1}$ and $r\in[0,T]$,
\[
\partial_{y_{i}}\psi(y,\eta)\chi^{(1)}(x,\eta,r)\chi^{(2)}(x',\eta,r)a_{ij}(\eta)\vr^{s,(1)}(x,y,\eta,r)\partial_{x'_{j}}\vr^{s,(2)}(x',y,\eta,r)\to0
\]
dominated convergence finishes the proof.
\end{proof}

\subsection{The hyperbolic errors\label{sec:hyp_err}}~We first provide the following  cancellation result.

\begin{lem} \label{lem:err_est}Let $\vr_{\ve,\d}^{(1)}$, $\vr_{\ve,\d}^{(2)}$ be as in \eqref{eq:test-functions} with $z^{(1)},z^{(2)}\in C_{0}([0,T];\R^N)$ and fix $\ve>0$, $\psi\in C_{c}^{\infty}(\R^{N+1})$. There exists a bounded function $c_{\psi,\d}:\R\to\R$ so that $\lim_{\d\to0}c_{\psi,\d}(\xi')=0$ for all $\xi'$ and
\begin{align*}
 & \int\Big|\int\Big(\psi(y,\eta)(\vr_{\ve,\d}^{(1)}(x,y,\xi,\eta,t)\partial_{\xi'}\vr_{\ve,\d}^{(2)}(x',y,\xi',\eta,t)\\
 & +\partial_{\xi}\vr_{\ve,\d}^{(1)}(x,y,\xi,\eta,t)\vr_{\ve,\d}^{(2)}(x',y,\xi',\eta,t))\Big)dyd\eta\Big|dxd\xi\le\frac{C}{\ve}|z_{t}^{(2)}-z_{t}^{(1)}|+\frac{c_{\psi,\d}(\xi')}{\ve}+(\|\partial_{\eta}\psi(\cdot,\cdot)\|_{C(\R^{N})}\ast\vr_{\d}^{v})(\xi').
\end{align*}
\end{lem} 

\begin{proof} We have
\begin{align*}
 & \vr_{\ve,\d}^{(1)}(x,y,\xi,\eta,t)\partial_{\xi'}\vr_{\ve,\d}^{(2)}(x',y,\xi',\eta,t)+\partial_{\xi}\vr_{\ve,\d}^{(1)}(x,y,\xi,\eta,t)\vr_{\ve,\d}^{(2)}(x',y,\xi',\eta,t)\\[0.5mm]
= & \vr_{\ve}^{s,(1)}(x,y,\xi,t)\vr_{\d}^{v}(\xi-\eta)\partial_{\xi'}[\vr_{\ve}^{s,(2)}(x',y,\xi',t)\vr_{\d}^{v}(\xi'-\eta)]\\[0.5mm]
 & +\partial_{\xi}[\vr_{\ve}^{s,(1)}(x,y,\xi,t)\vr_{\d}^{v}(\xi-\eta)]\vr_{\ve}^{s,(2)}(x',y,\xi',t)\vr_{\d}^{v}(\xi'-\eta)\\[0.5mm]
= & \vr_{\d}^{v}(\xi-\eta)\vr_{\d}^{v}(\xi'-\eta)\Big(\vr_{\ve}^{s,(1)}(x,y,\xi,t)\partial_{\xi'}\vr_{\ve}^{s,(2)}(x',y,\xi',t)+\partial_{\xi}\vr_{\ve}^{s,(1)}(x,y,\xi,t)\vr_{\ve}^{s,(2)}(x',y,\xi',t)\Big)\\
 & +\vr_{\ve}^{s,(1)}(x,y,\xi,t)\vr_{\ve}^{s,(2)}(x',y,\xi',t)\big(\partial_{\xi}\vr_{\d}^{v}(\xi-\eta)\vr_{\d}^{v}(\xi'-\eta)+\vr_{\d}^{v}(\xi-\eta)\partial_{\xi'}\vr_{\d}^{v}(\xi'-\eta)\big).
\end{align*}
We first note that
\begin{align*}
 & \int\psi(y,\eta)\vr_{\ve}^{s,(1)}(x,y,\xi,t)\vr_{\ve}^{s,(2)}(x',y,\xi',t)\big(\partial_{\xi}\vr_{\d}^{v}(\xi-\eta)\vr_{\d}^{v}(\xi'-\eta)+\vr_{\d}^{v}(\xi-\eta)\partial_{\xi'}\vr_{\d}^{v}(\xi'-\eta)\big)d\eta\\
 & =\int\partial_{\eta}\psi(y,\eta)\vr_{\ve}^{s,(1)}(x,y,\xi,t)\vr_{\ve}^{s,(2)}(x',y,\xi',t)\vr_{\d}^{v}(\xi-\eta)\vr_{\d}^{v}(\xi'-\eta)d\eta
\end{align*}
and, after integrating, we get 
\begin{align*}
 & \int\Big|\int\psi(y,\eta)\vr_{\ve}^{s,(1)}(x,y,\xi,t)\vr_{\ve}^{s,(2)}(x',y,\xi',t)\big(\partial_{\xi}\vr_{\d}^{v}(\xi-\eta)\vr_{\d}^{v}(\xi'-\eta)+\vr_{\d}^{v}(\xi-\eta)\partial_{\xi'}\vr_{\d}^{v}(\xi'-\eta)\big)dyd\eta\Big|dxd\xi\\
 & \le\int|\partial_{\eta}\psi(y,\eta)|\vr_{\ve}^{s,(1)}(x,y,\xi,t)\vr_{\ve}^{s,(2)}(x',y,\xi',t)\vr_{\d}^{v}(\xi-\eta)\vr_{\d}^{v}(\xi'-\eta)dyd\eta dxd\xi\\
 & =\int|\partial_{\eta}\psi(y,\eta)|\vr_{\ve}^{s,(2)}(x',y,\xi',t)\vr_{\d}^{v}(\xi-\eta)\vr_{\d}^{v}(\xi'-\eta)dyd\eta d\xi\\
 & \le\int\|\partial_{\eta}\psi(\cdot,\eta)\|_{C(\R^{N})}|\vr_{\d}^{v}(\xi-\eta)\vr_{\d}^{v}(\xi'-\eta)d\eta d\xi
  =\int\|\partial_{\eta}\psi(\cdot,\eta)\|_{C(\R^{N})}\vr_{\d}^{v}(\xi'-\eta)d\eta
 =(\|\partial_{\eta}\psi(\cdot,\cdot)\|_{C(\R^{N})}\ast\vr_{\d}^{v})(\xi').
\end{align*}
Next we observe that
\[
\partial_{\xi}\vr_{\ve}^{s,(i)}(x,y,\xi,t)=f'(\xi)z_{t}^{(i)}\partial_{y}\vr_{\ve}^{s,(i)}(x,y,\xi,t),
\]
and, hence,
\begin{align}
 & \vr_{\ve}^{s,(1)}(x,y,\xi,t)\partial_{\xi'}\vr_{\ve}^{s,(2)}(x',y,\xi',t)  +\partial_{\xi}\vr_{\ve}^{s,(1)}(x,y,\xi,t)\vr_{\ve}^{s,(2)}(x',y,\xi',t)\nonumber \\[0.5mm]
 &= \vr_{\ve}^{s,(1)}(x,y,\xi,t)f'(\xi')z_{t}^{(2)}\partial_{y}\vr_{\ve}^{s,(2)}(x',y,\xi',t)
 +f'(\xi)z_{t}^{(1)}\partial_{y}\vr_{\ve}^{s,(1)}(x,y,\xi,t)\vr_{\ve}^{s,(2)}(x',y,\xi',t) \label{eq:err_est} \\[0.5mm]
 &= \vr_{\ve}^{s,(1)}(x,y,\xi,t)f'(\xi)z_{t}^{(2)}\partial_{y}\vr_{\ve}^{s,(2)}(x',y,\xi',t)
  +f'(\xi)z_{t}^{(1)}\partial_{y}\vr_{\ve}^{s,(1)}(x,y,\xi,t)\vr_{\ve}^{s,(2)}(x',y,\xi',t)\nonumber \\
 & -\vr_{\ve}^{s,(1)}(x,y,\xi,t)(f'(\xi)-f'(\xi'))z_{t}^{(2)}\partial_{y}\vr_{\ve}^{s,(2)}(x',y,\xi',t).\nonumber 
\end{align}
For the last term we have the estimate
\begin{align*}
 & \int\Big|\int\psi(y,\eta)\vr_{\d}^{v}(\xi-\eta)\vr_{\d}^{v}(\xi'-\eta)\vr_{\ve}^{s,(1)}(x,y,\xi,t)(f'(\xi)-f'(\xi'))z_{t}^{(2)}\partial_{y}\vr_{\ve}^{s,(2)}(x',y,\xi',t)dyd\eta\Big|dxd\xi\\
 & \le\frac{|z_{t}^{(2)}|}{\ve}\int\int\psi(y,\eta)\vr_{\d}^{v}(\xi'-\eta)\vr_{\d}^{v}(\xi-\eta)|f'(\xi)-f'(\xi')|d\eta d\xi=:\frac{c_{\psi,\d}(\xi')}{\ve}.
\end{align*}
Since $f'$ is continuous, $c_{\psi,\d}$ is bounded and $\lim_{\d\to0}c_{\psi,\d}=0$. The conclusion now follows from the following estimate 
\begin{align*}
 & \int|\int\psi(y,\eta)\vr_{\d}^{v}(\xi-\eta)\vr_{\d}^{v}(\xi'-\eta)\big[\vr_{\ve}^{s,(1)}(x,y,\xi,t)f'(\xi)z_{t}^{(2)}\partial_{y}\vr_{\ve}^{s,(2)}(x',y,\xi',t)\\
 & +\psi(y,\eta)f'(\xi)z_{t}^{(1)}\partial_{y}\vr_{\ve}^{s,(1)}(x,y,\xi,t)\vr_{\ve}^{s,(2)}(x',y,\xi',t)\big]dyd\eta|dxd\xi\\
 & \le|z_{t}^{(1)}-z_{t}^{(2)}|\int\psi(y,\eta)\vr_{\d}^{v}(\xi-\eta)\vr_{\d}^{v}(\xi'-\eta)\vr_{\ve}^{s,(1)}(x,y,\xi,t)|\partial_{y}\vr_{\ve}^{s,(2)}|(x',y,\xi',t)|f'|(\xi)dyd\eta dxd\xi\\
 & =|z_{t}^{(1)}-z_{t}^{(2)}|\int\psi(y,\eta)\vr_{\d}^{v}(\xi-\eta)\vr_{\d}^{v}(\xi'-\eta)|\partial_{y}\vr_{\ve}^{s,(2)}|(x',y,\xi',t)|f'|(\xi)dyd\eta d\xi\\
 & \le\frac{|z_{t}^{(1)}-z_{t}^{(2)}|}{\ve}\int\|\psi(\cdot,\eta)\|_{C(\R^{N})}\vr_{\d}^{v}(\xi-\eta)\vr_{\d}^{v}(\xi'-\eta)|f'|(\xi)d\eta d\xi\\
 & =\frac{|z_{t}^{(1)}-z_{t}^{(2)}|}{\ve}\int\|\psi(\cdot,\eta)\|_{C(\R^{N})}\vr_{\d}^{v}(\xi-\eta)|f'|(\xi)d\xi\vr_{\d}^{v}(\xi'-\eta)d\eta\\
 & \le \sup_{\supp(\eta \to \|\psi(\cdot,\eta)\|_{C(\R^N)}) + [-\d,\d]} |f'|\ve^{-1} |z_{t}^{(1)}-z_{t}^{(2)}|.
\end{align*}
\end{proof}
For the  estimate for $Err^{(i)},i=1,2$ we recall that  
\begin{align*}
Err_{\ve,\psi,\d}^{(1)}(r) & =-\int\Big(\psi(y,\eta)\sgn(\xi)(\partial_{\xi}\vr_{\ve,\d}^{(1)}(x,y,\xi,\eta,r)\vr_{\ve,\d}^{(2)}(x',y,\xi',\eta,r)\\
 & +\vr_{\ve,\d}^{(1)}(x,y,\xi,\eta,r)\partial_{\xi'}\vr_{\ve,\d}^{(2)}(x',y,\xi',\eta,r))q^{(2)}(x',\xi',r)\Big)dxd\xi dx'd\xi'dyd\eta.
\end{align*}
Hence, using Lemma \ref{lem:err_est}, we find 
\begin{align*}
 & |Err_{\ve,\psi,\d}^{(1)}(r)|\\
 & \le\int q^{(2)}(x',\xi',r)\Big(\int\Big|\int\psi(y,\eta)(\partial_{\xi}\vr_{\ve,\d}^{(1)}(x,y,\xi,\eta,r)\vr_{\ve,\d}^{(2)}(x',y,\xi',\eta,r)\\
 & +\vr_{\ve,\d}^{(1)}(x,y,\xi,\eta,r)\partial_{\xi'}\vr_{\ve,\d}^{(2)}(x',y,\xi',\eta,r))dyd\eta\Big|dxd\xi\Big)dx'd\xi'\\
 & \le\int q^{(2)}(x',\xi',r)\Big(\frac{C}{\ve}|z_{t}^{(2)}-z_{t}^{(1)}|+\frac{c_{\psi,\d}(\xi')}{\ve}+(\|\partial_{\eta}\psi(\cdot,\cdot)\|_{C(\R^{N})}\ast\vr_{\d}^{v})(\xi')\Big)dx'd\xi',
\end{align*}
and, letting $\d\to0$, we get
\begin{align}
 & \lim_{\d\to0}\int_{s}^{t}|Err_{\ve,\psi,\d}^{(1)}(r)|dr\label{eq:Err_hyp_1}\\
 & \le\frac{C}{\ve}\|z^{(2)}-z^{(1)}\|_{C([s,t];\R^{N})}+\int_{s}^{t}\int q^{(2)}(x',\xi',r)\|\partial_{\eta}\psi(\cdot,\xi')\|_{C(\R^{N})}dx'd\xi'dr.\nonumber 
\end{align}
\no Now we present the proof of the error estimates for 
\[
\begin{cases}
Err_{\ve,\psi,\d}^{(1,2)}(t):=2\int\psi(y,\eta)\left(\chi^{(1)}(x,\xi,t)q^{(2)}(x',\xi',t)+\chi^{(2)}(x',\xi',t)q^{(1)}(x,\xi,t)\right)\\[1.5mm]
(\vr_{\ve,\d}^{(1)}(x,y,\xi,\eta,t)\partial_{\xi'}\vr_{\ve,\d}^{(2)}(x',y,\xi',\eta,t)+\partial_{\xi}\vr_{\ve,\d}^{(1)}(x,y,\xi,\eta,t)\vr_{\ve,\d}^{(2)}(x',y,\xi',\eta,t))dyd\eta dx'd\xi'dxd\xi,and
\end{cases}
\]
and we note that 
\begin{align}
 & |Err_{\ve,\psi,\d}^{(1,2)}(t)|\nonumber \\
 & \le2\int q^{(2)}(x',\xi',t)\Big(\int\Big|\int\big(\psi(y,\eta)\vr_{\ve,\d}^{(1)}(x,y,\xi,\eta,t)\partial_{\xi'}\vr_{\ve,\d}^{(2)}(x',y,\xi',\eta,t)\nonumber \\
 & +\partial_{\xi}\vr_{\ve,\d}^{(1)}(x,y,\xi,\eta,t)\vr_{\ve,\d}^{(2)}(x',y,\xi',\eta,t))\big)dyd\eta\Big|dxd\xi\Big)dx'd\xi'\label{eq:err12-1}\\
 & +2\int q^{(1)}(x,\xi,t)\Big(\int\Big|\int\big(\psi(y,\eta)(\vr_{\ve,\d}^{(1)}(x,y,\xi,\eta,t)\partial_{\xi'}\vr_{\ve,\d}^{(2)}(x',y,\xi',\eta,t)\nonumber \\
 & +\partial_{\xi}\vr_{\ve,\d}^{(1)}(x,y,\xi,\eta,t)\vr_{\ve,\d}^{(2)}(x',y,\xi',\eta,t))\big)dyd\eta\Big|dx'd\xi'\Big)dxd\xi.\nonumber 
\end{align}

Using Lemma \ref{lem:err_est} in \eqref{eq:err12-1} yields
\begin{align*}
|Err^{(1,2)}(t)|\le & 2\int q^{(2)}(x',\xi',t)\left(\frac{C}{\ve}|z_{t}^{(2)}-z_{t}^{(1)}|+\frac{c_{\psi,\d}(\xi')}{\ve}+(\|\partial_{\eta}\psi(\cdot,\cdot)\|_{C(\R^{N})}\ast\vr_{\d}^{v})(\xi')\right)dx'd\xi'\\
+ & 2\int q^{(1)}(x,\xi,t)\left(\frac{C}{\ve}|z_{t}^{(2)}-z_{t}^{(1)}|+\frac{c_{\psi,\d}(\xi)}{\ve}+(\|\partial_{\eta}\psi(\cdot,\cdot)\|_{C(\R^{N})}\ast\vr_{\d}^{v})(\xi)\right)dxd\xi.
\end{align*}
Letting $\d\to0$ and again using dominated convergence we get 
\begin{align}
\lim_{\d\to0}\int_{s}^{t}|Err^{(1,2)}(r)|dr\le & 2\int_{s}^{t}\int q^{(2)}(x',\xi',r)\left(\frac{C}{\ve}|z_{r}^{(2)}-z_{r}^{(1)}|+\|\partial_{\eta}\psi(\cdot,\xi')\|_{C(\R^{N})}\right)dx'd\xi'dr\nonumber \\
+ & 2\int_{s}^{t}\int q^{(1)}(x,\xi,r)\left(\frac{C}{\ve}|z_{r}^{(2)}-z_{r}^{(1)}|+\|\partial_{\eta}\psi(\cdot,\xi)\|_{C(\R^{N})}\right)dxd\xi dr.\label{eq:err12_bd}\\
\le & \frac{C}{\ve}\|z^{(2)}-z^{(1)}\|_{C([s,t];\R^{N})}+2\int_{s}^{t}\int q^{(2)}(x',\xi',r)\|\partial_{\eta}\psi(\cdot,\xi')\|_{C(\R^{N})}dx'd\xi'dr\nonumber \\
 & +2\int_{s}^{t}\int q^{(1)}(x,\xi,r)\|\partial_{\eta}\psi(\cdot,\xi)\|_{C(\R^{N})}dxd\xi dr\nonumber 
\end{align}

\section{The Regularity for the Fractional Heat semigroup}

We recall without a proof \cite[Lemma 9]{DV13}.
\begin{lem}
\label{lem:reg_fractional_heat}For $\g>0$, $\a\in(0,1]$, let $B_{\g}=\g(-\D)^{\a}$ on $\TT^{N}$ with corresponding semigroup $e^{-tB_{\g}}$. Then there exists a $C=C(N,n,m,\a,\b)$ such that, for all $1\le m\le n\le\infty$ and $\b\ge0,$
\[
\|(-\D)^{\frac{\b}{2}}e^{-tB_{\g}}\|_{L^{m}\to L^{n}}\le\frac{C}{|\g t|^{\frac{N}{2\a}(\frac{1}{m}-\frac{1}{n})+\frac{\b}{2\a}}}.
\]
\end{lem}

\section{The proof of Lemma~\ref{lem:integral_estimate-1}}\label{lemma}

We present here the proof of Lemma~\ref{lem:integral_estimate-1}.
\begin{proof}We compute the Fourier transform $\hat{\phi}$ of $\phi$. Using the elementary fact that\\
 $\int e^{-2\pi iz\cdot w}e^{-}\frac{|w|^{2}}{\d}dw=\sqrt{\d\pi}e^{-\d\pi^{2}|z|^{2}}$, we find 
\begin{align*}
\hat{\phi}(z) & =\int e^{-2\pi iz\cdot w}\phi(w)dw=\int e^{-2\pi iz\cdot w}e^{-\frac{|w|^{2}}{\d}}\int e^{ib(\xi)\cdot w-\d a(\xi)}f(\xi)d\xi dw\\
 & =\int\int e^{-2\pi i(z-\frac{1}{2\pi}b(\xi))\cdot w}e^{-\frac{|w|^{2}}{\d}}dwe^{-\d a(\xi)}f(\xi)d\xi=\int\sqrt{\d\pi}e^{-|\sqrt{\d}\pi(\frac{1}{2\pi}b(\xi)-z)|^{2}}e^{-\d a(\xi)}f(\xi)d\xi,
\end{align*}
and, hence,
\begin{align*}
|\hat{\phi}(z)|^{2} & =\left|\int\sqrt{\d\pi}e^{-|\sqrt{\d}\pi(\frac{1}{2\pi}b(\xi)-z)|^{2}-\d a(\xi)}f(\xi)d\xi\right|^{2}\\
 & \le\d\pi\int e^{-(|\sqrt{\d}\pi(\frac{1}{2\pi}b(\xi)-z)|^{2}+\d a(\xi))}d\xi\int e^{-|\sqrt{\d}\pi(\frac{1}{2\pi}b(\xi)-z)|^{2}-\d a(\xi)}f^{2}(\xi)d\xi.
\end{align*}
Next we use the assumption on $a,b$ which enters in the following straightforward estimate: 
\begin{align*}
\int e^{-\d(\frac{1}{4}|b(\xi)-2\pi z|^{2}+a(\xi))}d\xi & =\int_{0}^{\infty}\d e^{-\d\tau}|\{\xi:\frac{1}{4}|b(\xi)-2\pi z|^{2}+a(\xi)<\tau\}|d\tau\\
\le & \int_{0}^{\infty}\d e^{-\d\tau}|\{\xi:|b(\xi)-2\pi z|^{2}+a(\xi)<4\tau\}|d\tau\\
\le & \int_{0}^{\infty}\d e^{-\d\tau}\iota(4\tau)d\tau
\le  \frac{1}{4}\int_{0}^{\infty}e^{-\frac{\tau}{4}}\iota(\frac{\tau}{\d})d\tau.
\end{align*}
Hence, 
\begin{align*}
|\hat{\phi}(z)|^{2} & \le\frac{\d\pi}{4}\int_{0}^{\infty}e^{-\frac{\tau}{4}}\iota(\frac{\tau}{\d})d\tau\int e^{-|\sqrt{\d}\pi(\frac{1}{2\pi}b(\xi)-z)|^{2}-\d a(\xi)}f^{2}(\xi)d\xi.
\end{align*}
Integrating the above inequality in $z$ and using that $\int e^{-|\sqrt{\d}\pi(\frac{1}{2\pi}b(\xi)-z)|^{2}-\d a(\xi)}dz=\int e^{-\d\pi^{2}|z|^{2}-\d a(\xi)}dz\le\frac{1}{\sqrt{\d}}$ yields  
\begin{align*}
\int|\hat{\phi}(z)|^{2}dz & \le\frac{\sqrt{\d\pi}}{4}\int_{0}^{\infty}e^{-\frac{\tau}{4}}\iota(\frac{\tau}{\d})d\tau\|f\|_{2}^{2}.
\end{align*}
\end{proof}

\section{The Definition of the kinetic solution for smooth driving signals}\label{takis kinetic100}

For the convenience of the reader we recall here the definition of the kinetic solution to \eqref{eq:scl} for smooth signals given in \cite{CP03}. Note that this reference considers the special case of $z(t)=(t,\ldots,t)$ but the argument extends trivially to the case of $z\in C^1([0,T];\R^N)$. For the notation $\b_{ik}^\psi$ recall \eqref{eq:b_psi}.

\begin{defn}[Definition~2.2 in \cite{CP03}]\label{takis cp} 
Assume that $z\in C^1([0,T];\R^N)$. Then   $u\in C([0,T];L^{1}(\R^{N}))\cap L^{\infty}(\R^{N}\times[0,T])$ for each  $T>0$ is a kinetic solution to \eqref{eq:scl} if 

(i)~for all $k\in\{1,\dots,N\}$ and all nonnegative $\psi\in \mathcal D(\R)$, 
\begin{equation}\label{takis cp1}
\sum_{i=1}^{N}\partial_{x_{i}}\b_{ik}^\psi (u)\in L^{2}(\R^{N}\times[0,T]), 
\end{equation}
(ii)~for any nonnegative $\psi_1, \psi_2 \in \mathcal{D}(\R)$ and for all $k\in\{1,\dots,N\}$
\begin{equation}\label{takis cp2}
\sqrt {\psi_1(u(x,t))} \sum_{i=1}^{N}\partial_{x_{i}}\b_{ik}^{\psi_2} (u(x,t))=\sum_{i=1}^{N}\partial_{x_{i}}\b_{ik}^{\psi_1\psi_2} (u(x,t)),
\end{equation}
(iii) there exists a non-negative bounded measure $m$ on $\R^N \times\R \times [0,T]$ such that  \eqref{eq:kinetic_form} holds in the sense of distributions, where $n$ is a non-negative measure on $\R^N \times\R \times [0,T]$ such that, for any nonnegative $\psi \in \mathcal D(\R)$,
\begin{equation}\label{takis cp3} 
\int \psi(\xi) dn(x,\xi,t)=\sum_{k=1}^{N}\left(\sum_{i=1}^{N}\partial_{x_{i}}\b_{ik}^\psi(u(x,t))\right)^{2}.
\end{equation}
\end{defn}

\subsection*{Acknowledgements}
Part of the work was completed while B.\ Gess was visiting the University of Chicago. Souganidis is supported by the NSF grant DMS-1266383. B. Gess acknowledges financial support by the DFG through the CRC 1283 "Taming uncertainty and profiting from randomness and low regularity in analysis, stochastics and their applications."

\bibliographystyle{abbrv}
\bibliography{refs}

\end{document}